\documentclass[11pt,a4paper]{article}

\usepackage{geometry}
\usepackage[english]{babel}

\usepackage[utf8]{inputenc}

\usepackage{csquotes}

\usepackage{booktabs}
\usepackage{multicol}

\usepackage{amsthm, amssymb}
\usepackage{thmtools, thm-restate}
\usepackage{mathtools}
\usepackage[mathscr]{euscript}
\usepackage{dsfont}

\usepackage[T1]{fontenc}
\usepackage{lmodern}
\usepackage{textcomp}
\usepackage{txfonts}

\usepackage{enumitem}	
\usepackage{microtype}	

\usepackage{graphicx}

\usepackage{fancyvrb}

\usepackage{epsf}

\usepackage{wrapfig}
\usepackage{graphics}
\usepackage[usenames]{xcolor}
\usepackage{float}

\usepackage{pdfpages}

\usepackage{tikz}
\usetikzlibrary{cd}
\usetikzlibrary{patterns}
\usetikzlibrary{matrix}

\usepackage{hyperref} 

\usepackage{mleftright}
\renewcommand{\left}{\mleft}
\renewcommand{\right}{\mright}

\definecolor{thmback}{rgb}{1,1,1} 
\definecolor{lightback}{rgb}{1,1,1} 
\declaretheoremstyle[
	spaceabove=7pt, spacebelow=7pt,
	headfont=\normalfont\itshape,
	notefont=\mdseries, notebraces={(}{)},
	bodyfont=\normalfont,
	postheadspace=1em
]{rem}
\declaretheoremstyle[
	spaceabove=7pt, spacebelow=7pt,
	headfont=\normalfont\bfseries,
	notefont=\mdseries, notebraces={}{},
	bodyfont=\normalfont,
	postheadspace=1em,
	shaded={bgcolor=lightback,padding=2mm,textwidth=0.98\textwidth},
]{prop}
\declaretheoremstyle[
	spaceabove=20pt, spacebelow=7pt,
	headfont=\normalfont\bfseries,
	notefont=\mdseries, notebraces={(}{)},
	bodyfont=\normalfont,
	postheadspace=1em,
	shaded={bgcolor=thmback,padding=2mm,textwidth=0.98\textwidth},
]{forsty}
\declaretheoremstyle[
	spaceabove=7pt, spacebelow=7pt,
	headfont=\normalfont\bfseries,
	notefont=\normalfont\bfseries, 
	notebraces={}{},
	bodyfont=\normalfont,
	postheadspace=0.5em,
	headpunct={:}
]{def}
\newlength{\nameadjust}
\declaretheoremstyle[
	spaceabove=7pt, spacebelow=7pt,
	headfont=\normalfont\bfseries,
	notefont=\normalfont\bfseries,  
	notebraces={}{},
	bodyfont=\normalfont\itshape,
	postheadspace=1em,
	shaded={bgcolor=thmback,padding=2mm,textwidth=0.98\textwidth},
	headpunct={:},
]{thmsty}
\declaretheoremstyle[
	spaceabove=7pt, spacebelow=7pt,
	headfont=\normalfont\bfseries,
	notefont=\mdseries, notebraces={(}{)},
	bodyfont=\normalfont\itshape,
	postheadspace=1em,
	shaded={bgcolor=thmback,padding=2mm,textwidth=0.98\textwidth},
]{mainthmsty}
	\declaretheorem[style=thmsty, numberwithin=section,name=Theorem]{theorem}
	\declaretheorem[style=thmsty, numberlike=theorem,name=Lemma]{lemma}

	\declaretheorem[style=prop, numberlike=theorem,name=Proposition]{proposition}
	
	\declaretheorem[style=thmsty, numberlike=theorem,name=Corollary]{corollary}
	\declaretheorem[style=def, numberlike=theorem, name=Definition]{definition}
	
	\declaretheorem[style=def, numbered=no, name=Question]{question}
	\declaretheorem[style=rem, numbered=no, name=Remark]{remark}
	
\pgfdeclarepatternformonly{my crosshatch dots}{\pgfqpoint{-1pt}{-1pt}}{\pgfqpoint{10pt}{10pt}}{\pgfqpoint{12pt}{12pt}}%
{
    \pgfpathcircle{\pgfqpoint{0pt}{0pt}}{.5pt}
    \pgfpathcircle{\pgfqpoint{6pt}{6pt}}{.5pt}
    \pgfusepath{fill}
}
	\pgfdeclarepatternformonly{my crosshatch dots sparse}{\pgfqpoint{-1pt}{-1pt}}{\pgfqpoint{20pt}{20pt}}{\pgfqpoint{24pt}{24pt}}%
{
    \pgfpathcircle{\pgfqpoint{0pt}{0pt}}{.5pt}
    \pgfpathcircle{\pgfqpoint{12pt}{12pt}}{.5pt}
    \pgfusepath{fill}
}
\newlength{\hatchspread}
\newlength{\hatchthickness}
\newlength{\hatchshift}
\newcommand{\hatchcolor}{}
\tikzset{hatchspread/.code={\setlength{\hatchspread}{#1}},
         hatchthickness/.code={\setlength{\hatchthickness}{#1}},
         hatchshift/.code={\setlength{\hatchshift}{#1}},
         hatchcolor/.code={\renewcommand{\hatchcolor}{#1}}}
\tikzset{hatchspread=3pt,
         hatchthickness=0.4pt,
         hatchshift=0pt,
         hatchcolor=black}
\pgfdeclarepatternformonly[\hatchspread,\hatchthickness,\hatchshift,\hatchcolor]
   {custom north west lines}
   {\pgfqpoint{\dimexpr-2\hatchthickness}{\dimexpr-2\hatchthickness}}
   {\pgfqpoint{\dimexpr\hatchspread+2\hatchthickness}{\dimexpr\hatchspread+2\hatchthickness}}
   {\pgfqpoint{\dimexpr\hatchspread}{\dimexpr\hatchspread}}
   {
    \pgfsetlinewidth{\hatchthickness}
    \pgfpathmoveto{\pgfqpoint{0pt}{\dimexpr\hatchspread+\hatchshift}}
    \pgfpathlineto{\pgfqpoint{\dimexpr\hatchspread+0.15pt+\hatchshift}{-0.15pt}}
    \ifdim \hatchshift > 0pt
      \pgfpathmoveto{\pgfqpoint{0pt}{\hatchshift}}
      \pgfpathlineto{\pgfqpoint{\dimexpr0.15pt+\hatchshift}{-0.15pt}}
    \fi
    \pgfsetstrokecolor{\hatchcolor}
    \pgfusepath{stroke}
   }
\pgfdeclarepatternformonly[\hatchspread,\hatchthickness,\hatchshift,\hatchcolor]
   {custom north east lines}
   {\pgfqpoint{\dimexpr-2\hatchthickness}{\dimexpr-2\hatchthickness}}
   {\pgfqpoint{\dimexpr\hatchspread+2\hatchthickness}{\dimexpr\hatchspread+2\hatchthickness}}
   {\pgfqpoint{\dimexpr\hatchspread}{\dimexpr\hatchspread}}
   {
    \pgfsetlinewidth{\hatchthickness}
    \pgfpathmoveto{\pgfqpoint{\dimexpr\hatchshift-0.15pt}{-0.15pt}}
    \pgfpathlineto{\pgfqpoint{\dimexpr\hatchspread+0.15pt}{\dimexpr\hatchspread-\hatchshift+0.15pt}}
    \ifdim \hatchshift > 0pt
      \pgfpathmoveto{\pgfqpoint{-0.15pt}{\dimexpr\hatchspread-\hatchshift-0.15pt}}
      \pgfpathlineto{\pgfqpoint{\dimexpr\hatchshift+0.15pt}{\dimexpr\hatchspread+0.15pt}}
    \fi
    \pgfsetstrokecolor{\hatchcolor}
    \pgfusepath{stroke}
   }
\pgfdeclarepatternformonly[\starspread,\starsize]
  {custom stars pattern}
  {\pgfpointorigin}
  {\pgfqpoint{\starspread}{\starspread}}
  {\pgfqpoint{\starspread}{\starspread}}
  {
   \pgftransformshift{\pgfqpoint{\starsize}{\starsize}}
   \pgfpathmoveto{\pgfqpointpolar{18}{\starsize}}
   \pgfpathlineto{\pgfqpointpolar{162}{\starsize}}
   \pgfpathlineto{\pgfqpointpolar{306}{\starsize}}
   \pgfpathlineto{\pgfqpointpolar{90}{\starsize}}
   \pgfpathlineto{\pgfqpointpolar{234}{\starsize}}
   \pgfpathclose%
   \pgfusepath{fill}
  }

\pgfdeclarelayer{bg}    
\pgfsetlayers{bg,main} 

\newcommand{\supp}{\ensuremath{\operatorname{supp}}}

\newcommand{\defeq}{\mathrel{\vcentcolon =}}

\newcommand{\eps}{\ensuremath{\epsilon}}

\newcommand{\N}{\ensuremath{\mathds N}}	

\renewcommand{\H}{\Hil}
\newcommand{\Cpx}{\ensuremath{\mathds C}}

\newcommand{\B}{\ensuremath{\mathcal B}} 
\newcommand{\Cs}{\ensuremath{C^{\ast}}}

\newcommand{\Csalgebra}{\Cs-algebra}

\renewcommand{\star}{\ensuremath{^\ast}}
\newcommand{\starhomomorphism}{\star-homomorphism}
\newcommand{\starisomorphism}{\star-isomorphism}
\newcommand{\weakstar}{weak$^\ast$}

\renewcommand{\phi}{\ensuremath{\varphi}} 
\newcommand{\1}{\ensuremath{\mathds{1}}} 

\renewcommand{\ref}[1]{\textup{\ref{#1}}}

\usepackage[draft]{todonotes}   	

\newcommand{\etale}{\text{étale}}

\newcommand{\Csr}{\ensuremath{\Cs_{\text{r}}}} 

\newcommand{\A}{\ensuremath{\mathcal A}}
\renewcommand{\B}{\ensuremath{\mathcal B}}
\newcommand{\Ind}{\ensuremath{\mathop{\textup{Ind}}}}
\newcommand{\Calgebra}{\ensuremath{C_0( \Gzero )}\text{-algebra}}
\newcommand{\G}{\ensuremath{\mathcal G}}
\renewcommand{\H}{\ensuremath{\mathcal H}}
\newcommand{\Gzero}{\ensuremath{\G^{\left( 0 \right)}}}
\newcommand{\Hom}{\ensuremath{\textup{Hom}}}
\newcommand{\lift}{\ensuremath{\textup{rInd}\,}}
\newcommand{\res}{\ensuremath{\textup{res}\,}}

\newcommand{\intermediaryAlgebra}{\ensuremath{\Cs\left( \tilde \pi\left( C\left( \tilde X \right) \right) \right)}}
\renewcommand{\intermediaryAlgebra}{\ensuremath{E}}

\newcommand{\Galgebra}{\G-algebra}
\newcommand{\Halgebra}{\H-algebra}
\newcommand{\Gzeroalgebra}{\Gzero-algebra}

\title{The Furstenberg Boundary of a Groupoid}
\date{\empty}
\author{Clemens Borys
\thanks{Supported by a grant from the Danish Council for Independent Research, Natural Sciences.
  Part of this work was carried out 
  while visiting the Institute for Pure and Applied Mathematics (IPAM), 
  which is supported by the National Science Foundation.
}}

\begin{document}

\pagenumbering{arabic}
\setcounter{page}{1}

\maketitle

\begin{abstract}
  We define the Furstenberg boundary of a locally compact Hausdorff \etale{} groupoid,
  generalising the Furstenberg boundary for discrete groups,
  by providing a construction of a groupoid-equivariant injective envelope.
  Using this injective envelope, we establish the absence of recurrent amenable
  subgroups in the isotropy as a sufficient criterion for the intersection
  property of a locally compact Hausdorff \etale{} groupoid with compact unit space and
  no fixed points. This yields a criterion for {$C^*$}-simplicity of minimal
  groupoids.
\end{abstract}

\section{Introduction}
Introduced by Hamana \cite{HamanaCstar, HamanaOperatorSystems, HamanaDynamical} in the 1980s, 
injective envelopes in dynamical categories of operator systems
have recently gained in popularity,
after Kalantar and Kennedy \cite{KennedyBoundaries} 
identified the Furstenberg boundary $\partial_F\Gamma$ of a discrete group $\Gamma$
with the spectrum of the $\Gamma$-equivariant injective envelope
of the trivial $\Gamma$-algebra $\Cpx$.
A series of breakthrough results
of Breuillard, Kalantar, Kennedy, and Ozawa \cite{KennedyTrace}
and Kennedy \cite{KennedyIntrinsic}
then lead to the solution of the unique trace and \Cs-simplicity problems for discrete groups
via the characterisation of their actions on the boundary.
Consequently, several authors set forth to generalise Hamana's construction
to obtain Furstenberg-type boundaries in different settings,
in the hope of providing similarly powerful tools, 
compare for example Bearden--Kalantar \cite{KalantarUnitaryBoundary}
and Monod \cite{MonodFurstenberg}.

While the works above remain in the realm of groups, 
we aim to provide an injective envelope construction
that generalises the Furstenberg boundary 
to locally compact Hausdorff \etale{} groupoids with compact unit space.
Such groupoids may themselves be thought of as a generalisation 
of discrete groups acting on topological spaces,
where the group is allowed to ``vary'' over the space.

We introduce an induction functor in Section~\ref{sec:induction},
allowing us to transport injective operator systems 
to the category of operator systems with groupoid action,
and enabling us to construct the boundary of a given groupoid in said category 
in Section~\ref{sec:envelope}.
Consequently, we examine some properties of the boundaries 
thus defined in Section~\ref{sec:properties}
and finally apply them to obtain a sufficient condition 
for \Cs{}-simplicity of an \etale{} groupoid
in Section~\ref{sec:simplicity}.

The author is grateful to his supervisors M.\ Musat and M.\ Rørdam
for their continued help and support,
as well as to R.\ Meyer 
for some enlightening conversations.

\textbf{Remarks.} Shortly after we published our first pre-print,
Naghavi \cite{naghaviMinimalActions} obtained a dynamic description of Kawabe's boundaries
for actions of countable, discrete groups on minimal spaces.
Furthermore, she gave an explicit characterisation of the boundary in the case where the space acted upon is finite.
We remark that her boundary coincides with ours in the case of transformation groupoids and moreover any minimal Hausdorff \etale{} groupoids with finite (or, more generally, discrete) unit space
is automatically a transformation groupoid.

\section{Preliminaries}\label{sec:preliminaries}
We review some terminology related to 
\etale{} groupoids and their \Cs-algebras.
For more details see Sims' introductory notes \cite{SimsNotes}.

For a groupoid $\G$ we let $\Gzero$ denote its unit space,
and for every unit $u\in\Gzero$ we let
$\G^u \defeq r^{-1}(u)$ and
$\G_u \defeq s^{-1}(u)$ denote its range and source fibres
consisting respectively of all arrows ending or starting at $u$.
Likewise, let $\G_u^v \defeq \G_u \cap \G^v$ denote the set of all arrows 
with fixed source $u$ and range $v$.
The special case $\G_u^u$ is called the \emph{isotropy} at $u$.
We call $\G$ \emph{principal} if the only isotropy elements are the units themselves,
that is if $\G_u^u = \{u\}$ for all units $u$.

A Hausdorff topological groupoid is \emph{\etale{}} if the range map is a local homeomorphism,
in which case all fibres $\G_u$ and $\G^u$ are discrete in the subspace topology.
As such, \etale{} groupoids amongst topological groupoids 
generalise discrete groups amongst topological groups.
If $\G$ is \etale{} and its unit space is a totally disconnected space, 
we call the groupoid \emph{ample}.
A \emph{bisection} of $\G$ is a subset $B\subseteq \G$ such that
the restrictions of the source and range maps to $\G$ are injective.
In an \etale{} groupoid there is a basis of the topology
consisting of open bisections.
If it is ample, this basis may even be chosen to consist of clopen bisections.
Furthermore an \etale{} groupoid is \emph{topologically principal},
if the set of units at which the isotropy is trivial is dense in its unit space.
Unless otherwise stated, we assume all groupoids to be \etale{} and to have compact unit space.

Just as for groups, 
we can complete the algebra of continuous, compactly supported functions $C_c(\G)$
on an \etale{} groupoid $\G$ to a \Cs-algebra.
The \emph{reduced \Cs-algebra} of the groupoid will be the completion  of $C_c(\G)$
in the norm induced by a regular representation of $\G$.
To simplify notation we may speak of the ``regular representation at unit $u$''
when referring to the representation $\pi_u$ of $C_c(\G)$ on $l^2(\G_u)$ given by
\[
  \pi_u(f) \delta_\gamma 
  = \sum\limits_{\alpha \in \G_{r(\gamma)}} f(\alpha) \delta_{\alpha\gamma}.
\]
Note that the source fibres are discrete and so the sum above is finite.

While the action of a group on a topological space or \Cs-algebra
can be readily defined as a homomorphism into the appropriate automorphism group,
the fibred nature of a groupoid makes its actions much more delicate
and the objects acted upon will in general have to be fibred over its unit space.
Below we will introduce bundles of \Cs-algebras and actions of groupoids on such bundles
following the exposition \cite{GoehleThesis} of Goehle.
Recall that
for a locally compact Hausdorff space $X$,
a \emph{$C_0\left( X \right)$-algebra} is a \Cs-algebra $A$
equipped with a \starhomomorphism{} $\Phi$ from $C_0\left( X \right)$
into the centre of the multiplier algebra of $A$
such that $\Phi$ is non-degenerate in the sense that
$\Phi\left( C_0\left( X \right) \right)A $
is dense in $A$.
For $f\in C_0\left( X \right)$ and $a\in A$
we will often suppress $\Phi$ in the notation
and write $fa$ instead of $\Phi(f)a$.
There is a one-to-one correspondence between $C_0\left( X \right)$-algebras
and upper-semicontinuous bundles of \Cs-algebras over $X$:
see for example Williams \cite[Appendix C]{WilliamsCrossed}.
Therefore, we will not define such bundles separately,
but instead explain how to interpret $C_0\left( X \right)$-algebras as a bundle.

Let $A$ be a $C_0\left( X \right)$-algebra and take any $x\in X$.
Then $C_0\left( X\setminus x \right)$ is an ideal in $C_0\left( X \right)$
and consequently $I_x \defeq \overline{C_0\left( X\setminus x \right) A}$
is a closed, two-sided ideal in $A$.
We denote by $A_x\defeq A/I_x$ the quotient of $A$ by the ideal $I_x$ and assemble these
as fibres into the bundle $\A \defeq \bigsqcup\limits_{x\in X}A_x$.
This bundle then carries a unique topology for which the bundle map $p\colon \A \rightarrow X$
sending any element of $A_x$ to $x$ is continuous and open
and the \Csalgebra{} 
\[
  \Gamma_0\left( X, \A \right)
  =
  \left\{ f\colon X \rightarrow \A \text{ continuous}
    \bigm|
    f(x)\in A_x,
    \|f(x)\|_{A_x} \xrightarrow{x\to \infty} 0
  \right\}
\]
of continuous sections in $\A$ vanishing at infinity
is isomorphic to $A$
by the isomorphism $A \rightarrow \Gamma_0(X, \A)$ 
sending $a\in A$ to the function $x \mapsto a + I_x \in A_x$.
In this way, we may pass from bundles $\A$ to $C_0( X )$-algebras $A$ and back
by forming the algebra of continuous sections
and by assembling the bundle fibrewise.

Let $\tau\colon Y\rightarrow X$ be a continuous map 
between two locally compact Hausdorff spaces
and $A$ a $C_0(X)$-algebra.
Then the \emph{pullback bundle} $Y \ast_\tau \A$ is given by
\[
  Y \ast_\tau \A \defeq
  \left\{ (y,a) \in Y \times \A \bigm|
  \tau(y) = p(a)\right\}
\]
with the relative topology of $Y \times \A$.
It is an upper-semicontinuous bundle over $Y$
with bundle map $q\left( y,a \right) = y$.

We are now ready to define the action of a groupoid on a \Cs-algebra fibred over its unit space.
Let $\G$ be a locally compact Hausdorff \etale{} groupoid
and $A$ a $C_0( \Gzero )$-algebra. 
An \emph{action} $\alpha$ of $\G$ on $A$
is a family of \starisomorphism s $\alpha_\gamma \colon A_{s(\gamma)} \rightarrow A_{r(\gamma)}$
for $\gamma \in \G$
such that $\alpha_{\eta\gamma} = \alpha_{\eta} \circ \alpha_{\gamma}$
whenever $\eta$ and $\gamma$ are composable
and such that the map
\[
  \G \ast_s \A
  \rightarrow \A
  \qquad \text{given by} \qquad
  \left( \gamma, a \right) \mapsto \alpha_\gamma(a)
\]
is continuous.
We then call $A$ a \Galgebra{}.
If there is no confusion about the action, we will abbreviate $\alpha_\gamma(a)$ to $\gamma.a$.

A $C_0( \Gzero )$-map $\phi\colon A \rightarrow B$
between two \Galgebra{}s
with $\G$-actions denoted by $\alpha$ and $\beta$
is called
\emph{$\G$-equivariant}
if the induced maps on the fibres satisfy
\[
  \phi_{r(\gamma)}\left( \alpha_\gamma(a) \right)
  = \beta_\gamma\left( \phi_{s(\gamma)}(a) \right)
\]
for all $\gamma\in \G$ and $a \in A_{s(\gamma)}$.

An \emph{operator system} is a \star-closed subspace of a unital \Cs-algebra 
that contains the unit.
\begin{definition}
  A $\G$-operator system $S$
  is an operator system in a unital \Galgebra{} $A$
  that is closed under both the action of $C_0(\Gzero)$
  and the action of $\G$,
  the latter meaning that 
  $\alpha_\gamma(S/I_{s(\gamma)}) \subseteq A_{r(\gamma)}$
  is contained in $S/I_{r(\gamma)}$.
\end{definition}

\section{Induction of Groupoid Actions}\label{sec:induction}
The first step towards constructing an injective envelope among the $\G$-operator systems
is to find a suitable, but possibly too large, injective object in this category.
As injectivity of operator systems is well-understood in the absence of an action,
we aim to obtain a $\G$-injective object by inducing the trivial action of $\Gzero$
to an action of $\G$, and transporting morphisms in a natural fashion.
We are grateful to R. Meyer for pointing out that this is the 
correct notion behind the construction of group-equivariant injective envelopes.

Let $\G$ be a Hausdorff \etale{} groupoid
that acts on a \Galgebra{} $W$,
and $\H \subseteq \G$ a closed subgroupoid
such that $\H^{(0)} = \Gzero$.
Then $W$ is also an \Halgebra{}
and we denote the resulting restriction functor
from \Galgebra{}s to \Halgebra{}s by $\res_\H^{\G}$.
As alluded to above, we set out to find a right adjoint to $\res_\H^{\G}$
between the categories of $\G$- and $\H$-operator systems with equivariant unital completely positive (ucp) maps as morphisms.
Namely, we seek to assign 
a \Galgebra{} $\textup{Ind}_\H^{\G}(A)$
to every \Halgebra{} $A$ 
such that for every \Galgebra{} $W$ there is a natural bijection
\begin{equation}
  \Hom_\G\left( W, \textup{Ind}_\H^{\G}(A) \right)
  \cong
  \Hom_\H\left( \res_\H^{\G}W, A \right).
  \label{eq:inductionIsAdjoint}
\end{equation}
To construct injective envelopes, we may restrict to the case where $\H = \Gzero$.

In his recent PhD-thesis \cite[Chapter 3]{Boenicke}, C. B\"onicke provides a method
for inducing groupoid-\Csalgebra{}s from subgroupoids,
but this construction does not provide a right adjoint to restriction.
In this section, we modify his construction 
to obtain an induction functor satisfying Equation \eqref{eq:inductionIsAdjoint}.

As $\H = \Gzero$ consists only of units, there is no additional data in the $\Gzero$-action
and any \Gzeroalgebra{} that we want to induce to a \Galgebra{}
will have no further structure than that of a $C(\Gzero)$-algebra.
Let $A$ be such a \Gzeroalgebra{}.
As before we let $\A$ denote the bundle over $\Gzero$ associated with $A$.
We may then form its pullback $\G \ast_s \A$
along the source map $s\colon \G\rightarrow \Gzero$,
which is a bundle over $\G$.
We define the induced \Csalgebra{}
\begin{equation}
  \Ind A \defeq \Gamma_\textup{b}\left( \G, \G \ast_s \A \right)
  \label{eq:inductionDefinition}
\end{equation}
to be the \emph{bounded} continuous sections from $\G$ into that pullback.

\begin{remark}
It is here that we deviate from \cite[page 15]{Boenicke}, 
where only sections \emph{vanishing at infinity} in the appropriate sense are considered:
compare condition (2) on page 15 of \cite{Boenicke}. 
Note that condition (1) on the same page of \cite{Boenicke} is trivial for $\H = \Gzero$, 
but could well be added to obtain a more general induction functor.
Contrary to \cite{Boenicke}, our definition of $\Ind A$ 
does not yield a $C_0( \G )\text{-algebra}$,
as the action of $C_0(\G)$ is degenerate unless $\G$ is compact.

\end{remark}

Pushforward along the range map gives an action of $C(\Gzero)$
on $\Ind A$ by central multipliers:
For $f \in \Gamma_b(\G, \G\ast_s \A)$ and $g \in C ( \Gzero )$,
we set 
\[gf \defeq (g\circ r) f.\]
This action is non-degenerate, as $C( \Gzero )$ is unital.
It is worth pointing out that the fibre of $\Ind A$ at a unit $u$ is \emph{not} given by 
$\Gamma_b(\G^u, \G^u \ast_s, \A)$ as in \cite{Boenicke},
which makes it more difficult to formulate a $\G$-action on the induced bundle.
This is alleviated by 
the fibre projection only depending on the restriction to a neighbourhood of $\G^{u}$,
as we will see in the following lemma.
Recall that for a unit $u\in \Gzero$ the fibre $(\Ind A)_u$ at $u$ of the bundle $\Ind A$
is given by $(\Ind A) / I_u$,
where
$I_u = \overline{C_0\left( \Gzero\setminus u \right) \Ind A}$.
\begin{lemma}
  Let $f,g \in \Ind A$ be such that their restrictions to $r^{-1}(V)$ coincide
  for a neighbourhood $V\subseteq \Gzero$ of a unit $u\in \Gzero$.
  Then $f + I_u = g + I_u$.
  \label{lem:inducedFibersAreLocal}
\end{lemma}
\begin{proof}
  Take $f$ and $g$ as above and pick $h \in C(\Gzero)$ 
  such that $h(u) = 0$ and $h \equiv 1$ outside of $V$.
  Such $h$ since $\Gzero$ is normal.
  Then $g = f + h (g-f)$ 
  while $h (g-f) \in C_0( \Gzero \setminus u ) \Ind A$,
  so $f-g \in I_u$.
\end{proof}
For $f \in \Ind A$ and $u \in \Gzero$
we will write $[f]_u$ for the fibre projection $f+I_u$.
\begin{proposition}
  If $A$ is a \Gzeroalgebra{}, then $\Ind A$ is a \Galgebra{}.
  \label{prop:inductionAction}
\end{proposition}
\begin{remark}
  In the spirit of \cite{Boenicke}, where fibre are associated with restrictions to $\G^u$, 
  the action should be by composition on the argument.
  However, as the fibre projection of a section $f\in \Ind A$ at a unit $u \in \Gzero$ 
  is not determined by the values on $\G^u$ alone, 
  we cannot act by the single element $\gamma\in \G$.
  Given that values on a \emph{neighbourhood} of $\G^u$ determine the fibres,
  we may instead choose a bisection around $\gamma$ 
  and use this to act on the argument.
  Using the locality of Lemma \ref{lem:inducedFibersAreLocal},
  this will turn out to be well-defined.
\end{remark}
\begin{proof}
  We have already seen that $\Ind A$ is a \Calgebra,
  so it remains to describe the $\G$-action.

  Let $f\in \Ind A$ be a section of the induced bundle,
  $u,v \in \Gzero$ be units,
  and $\gamma\in \G_u^v$ be an arrow with source $u$ and range $v$.
  Considering $[f]_u \in \left( \Ind A \right)_u$
  we want to define $\gamma.[f]_u \in \left( \Ind A \right)_v$.
  As $\G$ is \etale, we may pick an open neighbourhood $B$ of $\gamma$ that is a bisection.
  Then $ U = s(B)$ and $V = r(B)$ are open neighbourhoods of $u$ and $v$.
  We write $B.f$ for the section 
  \[ 
    B.f( \eta ) \defeq f( B^{-1}\eta )
    \quad \text{in} \quad
    \Gamma_b(r^{-1}(V), r^{-1}(V) \ast_s \A),
  \]
  where $\eta \in r^{-1}(V)$ and $B^{-1}\eta = \xi^{-1}\eta$ 
  for the unique $\xi\in B$ with $r(\xi) = r(\eta)$.
  In order to extend $B.f$ to a section in $\Ind A$, 
  we choose a function $h \in C(\Gzero)$ which is identically one on a neighbourhood of $v$ 
  and vanishes outside of $V$. 
  Such an $h$ exists by normality of $\Gzero$.
  Then $h (B.f)$ extends to a section on all of $\G$ that vanishes outside of $r^{-1}(V)$,
  and we set $\gamma.[f]_u = [h\cdot B.f]_v$.

  The resulting class is independent of the choice of $h$, 
  as two different choices $h_1$ and $h_2$ coincide on a neighbourhood $V'$ of $v$
  and therefore the extensions $h_1 (B.f)$ and $h_2 ( B.f)$ coincide on $r^{-1}(V')$.
  Hence the fibre projections coincide by Lemma \ref{lem:inducedFibersAreLocal}.

  Similarly, this is independent of the choice of bisection $B$:
  For two open bisections $B_1$ and $B_2$ containing $\gamma$, 
  the intersection $B_1 \cap B_2$ is still an open bisection containing $\gamma$,
  so we may assume $B_2 \subseteq B_1$. 
  But then $r(B_2) \subseteq r(B_1)$ 
  and we can choose the same $h$ for both $B_2$ and $B_1$.
  With this choice, $h ( B_1.f)$ and $h ( B_2.f)$ are equal as sections,
  even before passing to the fibre projection.

  For simplicity, we may from now on drop $h$ in the construction above 
  and assume that $B.f$ can be extended to a section on all of $\G$ without prior modification,
  as we can just intersect $B$ with the $r$-preimage of the neighbourhood of $v$ 
  on which $h$ is constant.

  To finish off the argument that $\gamma.[f]_u \defeq [B.f]_v$ is well-defined, 
  we show that it only depends on the $u$-fibre projection of $f$.
  Assuming that $[f]_u = [g]_u$, we may for every $\eps>0$ 
  find $h \in C_0( \Gzero \setminus u ) \cdot \Ind A$
  such that $g$ and $f+h$ are $\eps$-close.
  Writing $h = h_1 \cdot h_2$ with $h_1 \in C_0( \Gzero \setminus u)$ 
  and $h_2 \in \Ind A$,
  we may consider $B.h = B.h_1 \cdot B.h_2$, 
  where the action of $B$ on $\Gzero$ is given by the local homeomorphism $s(B) \rightarrow r(B)$ 
  via $r \circ (s|_B)^{-1}$.
  By the arguments above, we may assume that $B.h_1$ and $B.h_2$ can be extended 
  to functions on all of $\Gzero$ and $\G$, respectively.
  As then $B.h_1 \in C_0( \Gzero \setminus v )$, 
  we find that $[B.f]_v = [B.f + B.h]_v$.
  But $B.f + B.h$ and $B.g$ are $\eps$-close on $r^{-1}(V)$ since the action of $B$ is pointwise,
  and by extending as above with a cut-off function that is bounded by one, 
  we may modify these to two $\eps$-close functions on all of $\G$ 
  without changing their classes in the $v$-fibre.
  Therefore $[B.f]_v$ and $[B.g]_v$ are $\eps$-close for all $\eps>0$, hence equal.

  The $\G$-action satisfies $(\gamma\eta).[f]_u = \gamma.(\eta.[f]_u)$, 
  as we can compose open bisections where defined to obtain another open bisection.
  As the defined maps act pointwise, they are \starhomomorphism{}s,
  and as they are invertible we obtain \starisomorphism{}s.

  Finally, we check that the $\G$-action is continuous.
  That is, if for $f, f_\lambda \in \Ind A$ and $\gamma, \gamma_\lambda \in \G$ 
  we have $[f_\lambda]_{s(\gamma_\lambda)} \to [f]_{s(\gamma)}$ and $\gamma_\lambda \to \gamma$, 
  then we need to show that
  $\gamma_\lambda.[f_\lambda]_{s(\gamma_\lambda)} \to \gamma.[f]_{s(\gamma)}$.
  If $B$ is an open bisection containing $\gamma$, 
  then we will eventually have $\gamma_\lambda \in B$.
  Hence, we may use the same bisection to construct 
  $\gamma_\lambda.[f_\lambda]_{s(\gamma_\lambda)}$ and $\gamma.[f]_{s(\gamma)}$.
  Consequently, we have to show that $[B.f_\lambda]_{r(\gamma_\lambda)} \to [B.f]_{r(\gamma)}$,
  which is equivalent to 
  $\inf_r \|B.f_\lambda - B.f + h' \|_\infty \to 0$
  as $\lambda\to\infty$,
  where $\inf_r$ denotes the infimum taken over all 
  $h' \in C_0( \Gzero \setminus r(g_\lambda) )\cdot \Ind A$.
  We know, however, that $\inf_s \| f_\lambda - f + h \|_\infty \to 0$,
  where $\inf_s$ is the infimum taken over all 
  $h \in C_0( \Gzero \setminus s(\gamma_\lambda) )\cdot \Ind A$.
  By choosing $h'$ as $B.h$, we may bound
  \begin{equation}
    \inf\nolimits_r \|B.f_\lambda - B.f + h' \|_\infty 
    \leq \inf\nolimits_s \| f_\lambda - f + h \|_\infty \to 0 
    \label{eq:continuousActionByBisections}
  \end{equation}
  which yields the claim.
\end{proof}
\begin{remark}
  Note that we did not require $f$ to be continuous in the proof of the previous lemma.
  We will later turn the bounded, \emph{not necessarily continuous} sections
  $l^{\infty}\left( \G, \G \ast_s\A \right)$
  into a \Galgebra{} in the same way.
  \label{rem:boundedSectionsGAlgebra}
\end{remark}

The following notational remark is essential to avoid confusion:
\begin{remark}
For a section $a \in A = \Gamma( \Gzero, \A )$ 
we will denote by $a_u$ the value $a(u)$ of $a$ at the unit $u\in \Gzero$, 
which is the same as the projection $[a]_u$ of $a$ onto the appropriate fibre.
This is straightforward enough for $A$, but note that
we defined elements $f \in \Ind B$ in the induced \Csalgebra{} 
as functions on $\G$, not $\Gzero$.
Nevertheless, the associated bundle is of course still fibred over $\Gzero$, 
taking values in the appropriate quotients of $\Ind B$.
As such, while it makes sense to speak of $f(\gamma) \in B_{s(\gamma)}$,
this is \emph{not} the value of the section associated with $f$ at any given fibre,
even if $\gamma$ were a unit.
We will therefore refrain from writing $f_u$ for the projection 
$[f]_{u} \in \left( \Ind B \right)_u$
of $f$ in the fibre at a unit $u$,
as this class should not be confused with $f(u) \in B_{u}$.
\end{remark}

Let $A$ be a \Galgebra{} and $B$ a \Gzeroalgebra{}.
Below, we explain how
a $\Gzero$-\starhomomorphism{} $\phi\colon~A~\rightarrow~B$ 
lifts to a $\G$-\starhomomorphism{} $\psi\colon A \rightarrow \Ind B$.
First, we note that $A$ embeds into $\Ind A$ as \Galgebra{}s,
where we drop spelling out the restriction 
and denote $A$ as a $\G$- and a \Gzeroalgebra{} simultaneously:
Let
\[ 
  \iota \colon A 
  = \Gamma(\Gzero, \A) 
  \rightarrow \Gamma_b(\G, \G \ast_s \A) 
  = \Ind A 
\]
be given by sending a section $a$ 
to the section $\gamma \mapsto \alpha_\gamma^{-1}(a_{r(\gamma)})$,
where $\alpha$ is the $\G$-action on $\A$.
This is obviously a $C ( \Gzero )$-linear map, 
since multiplication is pointwise after the pushforward via the range map.
The given section is continuous, since it is given by a composition of continuous maps, namely
\begin{align*}
  \G 
  &\rightarrow \G \ast_r \Gzero 
  \rightarrow \G \ast_r \A 
  \rightarrow \G \ast_s \A \\
  \gamma
  &\mapsto (\gamma, r(\gamma))
  \mapsto (\gamma, a_{r(\gamma)})
  \mapsto (\gamma, \alpha_\gamma^{-1}(a_{r(\gamma)})),
\end{align*}
where the last map is continuous by continuity of the $\G$-action on $\A$.
Furthermore, $\iota$ is $\G$-equivariant, as $\G$ acts on the argument on $\Ind A$.
Note that all sections $\iota(a)$ have constant norm on range-fibres
and that this implies that $[\iota(a)]_u$ is determined by its values on $\G^u$,
in contrast to general sections in $\Ind A$.

Second, from a $\Gzero$-\starhomomorphism{} $\phi\colon A \rightarrow B$, 
we obtain a $\G$-\starhomomorphism{} $\Ind \phi$ mapping $\Ind A \rightarrow \Ind B$ as follows:
\begin{align*}
  \Ind \phi \colon \Ind A 
  = \Gamma_b(\G, \G\ast_s \A) &\rightarrow \Gamma_b(\G, \G\ast_s \B) 
  = \Ind B \\
  f &\mapsto \left( \gamma \mapsto \phi_{s\left( \gamma \right)}\left( f(\gamma) \right) \right).
\end{align*}
This is $C ( \Gzero )$-linear, 
since the appropriate fibre of $\phi$ is applied pointwise;
it is also $\G$-equivariant, 
since the map is pointwise and $\G$ acts on the argument.
The assigned section is obviously bounded, 
so it only remains to show that it is continuous.
Given a net $\gamma_\lambda \to \gamma$ in $\G$ and $f\in \Ind A$, 
we have $f(\gamma_\lambda) \to f(\gamma)$
and so for all sections $a \in A$ with $a_{s(\gamma)} = f(\gamma)$ we find that
$\| a_{s(\gamma_\lambda)} - f(\gamma_\lambda) \|_{A_{s(\gamma_\lambda)}} \to 0$.
Therefore
\[
  \left\| \left( \phi(a) \right)_{s(\gamma_\lambda)} 
  - \Ind \phi \left( f \right)(\gamma_\lambda) \right\|_{B_{s(\gamma_\lambda)}} 
  = 
  \left\| \phi_{s(\gamma_\lambda)}\left(a_{s(\gamma_\lambda)}\right) 
  - \phi_{s(\gamma_\lambda)}\left( f(\gamma_\lambda) \right) \right\|_{B_{s(\gamma_\lambda)}} 
  \leq 
  \left\| a_{s(\gamma_\lambda)} - f(\gamma_\lambda) \right\|_{A_{s(\gamma_\lambda)}}
\]
vanishes in the limit.
As $\phi$ is continuous, we have $\phi(a) \in B$ 
with $\left( \phi(a) \right)_{s(\gamma)} = \phi_{s(\gamma)}\left( f(\gamma) \right)$.
We may conclude that $\left( \Ind \phi(f) \right)(\gamma_\lambda)$ 
converges to $\left( \Ind \phi(f) \right)(\gamma)$.

Altogether, the adjoint isomorphism 
${\lift{}}\!\colon \Hom_{\Gzero}(\res A, B) \rightarrow \Hom_{\G}(A, \Ind B)$
is given by \[\lift \phi \defeq \Ind \phi \circ \iota\]
for $\phi \in \Hom_{\Gzero}\left( A, B \right)$.
Explicitly, a section $a \in A$ is mapped to the section given by
\begin{equation}
  \left( (\lift{}\phi)(a) \right)_{\gamma} 
  = \phi_{s(\gamma)}\left( \alpha_\gamma^{-1}\left( a_{r\left( \gamma \right)} \right) \right).
  \label{eq:explicitLift}
\end{equation}
We now see why we had to modify the induction procedure of \cite{Boenicke}:
The sections defined in Equation~\eqref{eq:explicitLift}
do not vanish at infinity.

Conversely, any $\G$-\starhomomorphism{ } $\psi\colon A \rightarrow \Ind B$ restricts to 
a $\Gzero$-\starhomomorphism{ } which we denote as $\res \psi\colon A \rightarrow B$ 
by restricting from $\G$ to $\Gzero$ as
\[\left( \res \psi(a) \right)_u \defeq  \psi(a) (u).\]
Extending $\phi$ and then restricting to $\res \lift \phi$ obviously gives back $\phi$, 
since $\phi_u( \alpha_u^{-1}\left( a_u \right) )=\phi_u(a_u)$.
On the other hand, we will see that restricting $\psi$ 
and then extending the result to $\lift \res \psi$ also gives back $\psi$.

For every $\gamma \in \G$ we obtain a \starhomomorphism{} 
$\text{ev}_\gamma\colon \Ind B \rightarrow B_{s(\gamma)}$
by evaluating a section in $\Ind B$ at $\gamma$.
As $\gamma \in \G^{r(\gamma)}$, 
this \starhomomorphism{} factors through the quotient map to the fibre at $r(\gamma)$
and satisfies a straightforward equivariance condition:
\begin{lemma}\label{lem:ev-equivariance}
  For $b_{r(\gamma)}\in B_{r(\gamma)}$ with $\gamma, \eta \in \G$ 
  where $s(\gamma) = r(\eta)$, we have 
  \[
    \text{ev}_\eta\left( \gamma^{-1}.\left( b_{r(\gamma)} \right) \right)
    = 
    \text{ev}_{\gamma\eta}\left( b_{r(\gamma)} \right).
  \]
\end{lemma}
\begin{proof}
  This comes down to $\gamma$ acting on the argument.
  If $b \in B$ is a section with value $b_{r(\gamma)}$ in the appropriate fibre
  and $S$ is an open bisection containing $\gamma^{-1}$,
  then any lift of $\gamma^{-1} b_{r(\gamma)}$ to a section $Sb$ in $\Ind B$ is
  given by
  $\left( Sb \right)(\eta) = b(S^{-1}\eta)$.
  on a neighbourhood of $\G^{s(\gamma)}$. 
  Hence for $\eta \in \G^{s(\gamma)}$ we find 
  \[
    \text{ev}_\eta \left( \gamma^{-1}b_{r(\gamma)} \right)
    = b(S^{-1}\eta) 
    = b(\gamma\eta) 
    = \text{ev}_{\gamma\eta}\left( b_{r(\gamma)} \right).
  \]
\end{proof}
Using the previous lemma,
restricting $\psi \in \Hom_\G\left( A, \Ind B \right)$ 
and then lifting the restriction gives
\begin{align*}
  \left( \lift \res \psi \right)(a)(\gamma)
  &=
  \left(\res \psi \right)_{s(\gamma)}
  \left( \gamma^{-1} \left( a_{r\left( \gamma \right)} \right) 
  \right)\\
  &=
  \text{ev}_{s(\gamma)}\left( \psi_{s(\gamma)}
    \left( \gamma^{-1} \left( a_{r\left( \gamma \right)} \right) \right) 
  \right) \\ 
  &=
  \text{ev}_{s(\gamma)}\left( 
    \gamma^{-1}\left(\psi_{r(\gamma)}\left( a_{r\left( \gamma \right)} \right) \right) 
  \right)\\
  &=
  \text{ev}_{\gamma} \left( 
    \psi_{r(\gamma)}\left( a_{r\left( \gamma \right)} \right) 
  \right)\\
  &= \psi(a)(\gamma).
\end{align*}
Therefore, $\lift\colon Hom_{\Gzero}\left( A, B \right) \cong \Hom_{\G}\left( A, \Ind B \right)$ 
is a bijection,
as the two constructions are inverse to each other.

We proceed to show that this isomorphism is natural.
That is, given a $\G$-\starhomomorphism{ } $j\colon A\rightarrow B$ 
between \Galgebra{}s $A$ and $B$,
as well as a \Gzeroalgebra{} $C$, the following diagram commutes:
\begin{equation}\label{cd:Induction-Naturality}
  \begin{tikzcd}
    \Hom_{\Gzero}\left( B, C \right)
    \ar[r, rightarrow, "\cong" below, "\lift"]
    \ar[d, rightarrow, "j^\ast"]
    & \Hom_\G\left( B, \Ind C \right)
    \ar[d, rightarrow, "j^\ast"]
    \\
    \Hom_{\Gzero}\left( A, C \right)
    \ar[r, rightarrow, "\cong" below, "\lift"]
    & \Hom_\G\left( A, \Ind C \right)
  \end{tikzcd}
\end{equation}
Given $\phi \in \Hom_{\Gzero}\left( B, C \right)$
we have to show that $\lift \left(\phi \circ j \right)$ 
equals $\left( \lift \phi \right)\circ j$.
Indeed we find, using equivariance of $j$, that
\begin{align*}
  \left( \lift\left( \phi \circ j \right)\right) (a)(\gamma)
  &= \left( \phi \circ j \right)_{s(\gamma)}
  \left( 
    \gamma^{-1} \left( a_{r(\gamma)} \right) 
  \right) \\
  &=  \phi_{s(\gamma)} \left( 
    \gamma^{-1} \left( j_{r(\gamma)}\left( a_{r(\gamma)} \right) \right) 
  \right) \\
  &= \left( \lift \phi \right)\left( j(a) \right) (\gamma),
\end{align*}
where the action of $\gamma^{-1}$ on $a_{r(\gamma)}$ and $b_{r(\gamma)}$
is the $\G$-action on $A$ or $B$, respectively.

So far, we have only considered \starhomomorphism{}s as morphisms
between \Galgebra{}s.
To apply our induction procedure in the construction of a $\G$-equivariant injective envelope,
we will need to broaden our scope to also include unital, completely positive,
$\G$-equivariant $C ( \Gzero )$-maps between \Galgebra{}s.

First we note that being positive is a fibre-wise condition:
\begin{lemma}
  \label{lem:positiveIsFiberwise}
  Let $\phi\colon A \rightarrow B$
  be a $C_0\left( X \right)$-map between two $C_0\left( X \right)$-algebras $A$ and $B$.
  Then $\phi$ is positive if and only if $\phi_x \colon A_x \rightarrow B_x$
  is positive for every $x$ in $X$.
  Likewise, $\phi\colon A\rightarrow B$
  is completely positive
  if and only if all $\phi_x\colon A_x\rightarrow B_x$
  are completely positive.
\end{lemma}
\begin{proof}
  We may embed $A$ into $\bigoplus\nolimits_{x\in X} A_x$, and as injective \starhomomorphism{}s 
  are order embeddings, the first statement follows.

  For the second half observe that
  $M_n(A)$ is a $C_0(X)$-algebra
  whose fibre can be understood as
  $\left( M_n(A) \right)_x \cong M_n\left( A_x \right)$
  and that under this identification
  $\left( \phi^{(n)} \right)_x = \left( \phi_x \right)^{(n)}$.
\end{proof}
Now, for a ucp $\Gzero$-map $\phi\colon A \rightarrow B$
from a \Galgebra{} $A$ to a \Gzeroalgebra{} $B$,
we may as above form 
\[
  \Ind \phi(f)(\gamma) \defeq \phi_{s(\gamma)}\left( f(\gamma) \right)
  \qquad\text{and}\qquad
  \lift \phi \defeq \Ind \phi \circ \iota.
\]
Without any modification to the arguments above, 
$\lift$ still is natural and an inverse to the restriction,
and $\lift \phi$ is ucp, 
since a section $f\in \Gamma_b\left( \G, \G\ast_s \A \right)$ is positive
if and only if $f\left( \gamma \right)$ is positive for all $\gamma \in \G$.
Hence, if $\phi\colon A \rightarrow B$ is completely positive 
then so are $\Ind \phi\colon \Ind A \rightarrow \Ind B$
and $\lift \phi\colon A \rightarrow \Ind B$.

\section{The Groupoid Furstenberg Boundary}\label{sec:envelope}
Let $\G$ be a locally compact Hausdorff \etale{} groupoid
with compact unit space $X \defeq \Gzero$.
Using the construction of $\G$-equivariant injective envelopes above, 
we may now find a $\G$-injective \Galgebra{} enveloping $C(X)$
in the category of $\G$-operator systems.
Note that this generalises the Furstenberg boundary of a discrete group $\Gamma$, 
where the unit space $X$ is a single point and the boundary coincides
with the spectrum of the $\Gamma$-equivariant injective envelope of $C(X) = \Cpx$.
The methods below may however be used more generally 
to construct groupoid-equivariant injective envelopes.

Consider the (non-dynamic) injective envelope 
$I \defeq I\left( C(X)  \right)$.
Let $\tilde X$ denote its spectrum by, so that 
$I = C( \tilde X )$.
As $C(X)$ embeds into $C(\tilde X)$, the latter is a $C(X)$-algebra
and it is furthermore injective among such:
For any two (unital) $C(X)$-algebras $V\subseteq W$, 
we may lift a ucp $C(X)$-map $\psi\colon V\rightarrow C( \tilde X )$
to a ucp map $W \rightarrow C( \tilde X )$ by disregarding the $C(X)$-structure,
and the result will necessarily be a $C(X)$-map
if all maps and algebras are unital,
as in that case the action of $C(X)$ on $V$ and $W$ is determined by a subalgebra of $V$
which lies in the multiplicative domain of $\psi$.

Therefore, the induced \Csalgebra{},
$\Ind  C( \tilde X ) $,
is a $\G$-injective \Galgebra{}:
Given two \Galgebra{}s $V\subseteq W$ 
and a $\G$-equivariant ucp map
$\psi \colon V\rightarrow \Ind C( \tilde X ) $, 
we may first restrict to $\res \psi\colon V \rightarrow C( \tilde X )$, 
then use the $\Gzero$-injectivity 
to extend the restriction to $\phi\colon W \rightarrow C( \tilde X )$
and consequently lift to a $\G$-\starhomomorphism{} 
$\lift \phi \colon 
W \rightarrow \Ind C( \tilde X ) $.
By naturality as in Diagram~\eqref{cd:Induction-Naturality}, this is the desired extension,
as lifting $\res\psi$ results in $\lift \res \psi = \psi$.

\begin{equation}
  \label{cd:injectiveExtensionByInduction}
  \begin{tikzcd}
    W
    \ar[dr, "\phi"description, dashed]
    \ar[drr, bend left, "\lift \phi"description]
    \\
    & C( \tilde X )
    & \Ind C( \tilde X )
    \ar[l, "\text{res}"description, gray]
    \\
    V
    \ar[ur, "\res \psi"description]
    \ar[uu, hookrightarrow]
    \ar[urr, bend right, "\psi = \lift \res \psi"description]
  \end{tikzcd}
\end{equation}

Having found a $\G$-injective $\G$-extension of the $\G$-\Cs-algebra $C(X)$,
we follow Hamana's  scheme \cite{HamanaDynamical}
to construct an injective envelope, 
adapted to a $\G$-equivariant setting.
The following definitions are $\G$-equivariant adaptations
of the corresponding notions in \cite{HamanaDynamical}.

\begin{definition}
  Let $A$ and $B$ be $\G$-\Cs-algebras with $A$ a sub-$\G$-\Cs-algebra of $B$.
  An \emph{$A$-seminorm} on $B$ is a seminorm $p$ on $B$, such that
  $p(b) = \|\phi(b)\|_B$ 
  for some $\G$-equivariant $C( \Gzero )$-linear ucp map 
  $\phi \colon B\rightarrow B$
  that restricts to the identity on $A$.
  \label{def:ASeminorm}
\end{definition}

The $A$-seminorms are partially ordered by the pointwise order
where an $A$-seminorm $p$ is dominated by an $A$-seminorm $q$ if
$p(b) \leq q(b)$ for all  $b \in B$.
We write $p \prec q$.

\begin{definition}
  Let $A$ and $B$ be $\G$-\Cs-algebras with $A$ a sub-$\G$-\Cs-algebra of $B$.
  A map $\phi$ as above which is furthermore idempotent is called 
  an $A$-\emph{projection} on $B$,
  that is, 
  a $\G$-equivariant $C( \Gzero )$-linear idempotent ucp map 
  $\phi\colon B\rightarrow B$
  that restricts to the identity on $A$.
  \label{def:AProjection}
\end{definition}

Note that despite the name the range of an $A$-projection is in general not equal to $A$, 
but merely contains it.
The $A$-projections are likewise partially ordered
by $\phi \leq \psi$
if
$\phi \circ \psi = \phi = \psi \circ \phi$.
To every $A$-projection $\phi$ on $B$ there is an associated $A$-seminorm $p$
defined by $p(b) = \|\phi(b)\|_B$ for $b\in B$.

\begin{definition}
  Let $A$ be a \Galgebra{} 
  and $\iota$ be a $\G$-equivariant embedding of $A$
  into another \Galgebra{} $B$.
  We then call $\left( B, \iota \right)$ a $\G$-\emph{extension} of $A$.
  The extension is said to be $\G$-\emph{injective}
  if $B$ is a $\G$-injective \Galgebra{}.
  \label{def:extension}
\end{definition}

\begin{definition}
  A $\G$-extension $(B, \iota)$ of $A$ is said to be \emph{$\G$-essential}
  if any ucp $\G$-map $\phi\colon B \rightarrow C$
  into a third \Galgebra{} $C$ is injective
  if $\phi\circ \iota$ is injective on $A$.

  The $\G$-extension is said to be \emph{$\G$-rigid}
  if the identity is the unique ucp $\G$-map $\psi\colon B\rightarrow B$
  that satisfies $\psi \circ \iota = \iota$.
  \label{def:extensionEssentialRigid}
\end{definition}

We will set out to show the existence of a minimal $A$-seminorm
from which we will construct a minimal $A$-projection.
The minimal $A$-projection then yields a $\G$-rigid $\G$-injective $\G$-extension 
by equipping its image with the Choi--Effros multiplication,
which turns its range into a \Cs-algebra 
without changing the complete order isomorphism class.
As any rigid injective extension is essential, 
this will be the desired $\G$-injective envelope:

\begin{definition}
  A \emph{$\G$-injective envelope} of a \Galgebra{}
  is a $\G$-extension
  which is both $\G$-injective
  and $\G$-essential.
  \label{def:injectiveEnvelope}
\end{definition}

Aiming to apply Zorn's lemma, 
we show that every decreasing net of $A$-seminorms has a lower bound.
Again we denote by $X \defeq \Gzero$ the unit space of a Hausdorff \etale{} groupoid $\G$
and by $I = C(\tilde X)$ the non-dynamic injective envelope of $C(X)$,
such that $\Ind I$ is a $\G$-injective \Galgebra.
\begin{lemma}
  Every decreasing net $p_i$ of $C(X)$-seminorms on $\Ind I$ has a lower bound.
  \label{lem:lowerBound}
\end{lemma}
\begin{proof}
  In order to take \weakstar-limits we embed $\Ind I$ into a von-Neumann algebra.
  Observe that the embedding $C(X) \hookrightarrow I = C( \tilde X )$
  yields a surjective continuous map $q\colon \tilde X \rightarrow X$
  and that the fibres of $I$ as a $C(X)$-algebra
  are of the form $C(q^{-1}(x))$.
  We consider $\ell^\infty(\tilde X)$ seen as a $C(X)$-algebra
  with the obvious map of $C\left( X \right)$ into its central multipliers,
  yielding a decomposition of the associated bundle into fibres as
  $\bigsqcup_{x \in X}\ell^{\infty}( q^{-1}\left( x \right) )$.
  We then denote by
  \[
    M \defeq 
    \ell^{\infty}\left( \G, 
      \G \ast_s \bigsqcup_{x \in X}\ell^{\infty}\left( q^{-1}\left( x \right) \right) 
    \right)
    \cong
    \ell^\infty \left( \left\{ 
        \left( \gamma, \tilde x \right) \in \G \times \tilde X
        \bigm|
        s(\gamma) = q(\tilde x)
    \right\} \right),
  \]
  the bounded sections in the pullback along $s$ that are not necessarily continuous.
  As noted before, $M$ can be equipped with a $\G$-action 
  in the same way that $\Ind I$ can, making it a \Galgebra{}.
  We can embed $\Ind I$ into $M$ as a \Galgebra{}
  by utilizing the inclusion of $I= C(\tilde X)$ into $\ell^\infty(\tilde X)$.
  We denote this inclusion $\Ind I \rightarrow M$ by $\kappa$.
  By $\G$-injectivity of $\Ind I$, we may lift the identity on $\Ind I$
  along the inclusion $\kappa\colon \Ind I \hookrightarrow M$
  to a ucp $\G$-map $E\colon M \rightarrow \Ind I$,
  which necessarily contains the image of $\Ind I$ in its multiplicative domain.
  Using this, we may take weak limits in $M$ and project them back down to $\Ind I$
  as follows.

  Let $p_i$ be the decreasing net of $C(X)$-seminorms on $\Ind I$, 
  and let $\phi_i\colon \Ind I \rightarrow \Ind I$ be the ucp $\G$-maps fixing $C(X)$
  associated with the seminorms defined by $p_i(x) = \|\phi_i(x)\|_{\Ind I}$.
  As the maps $\kappa\circ\phi_i\colon \Ind I \rightarrow M$ 
  are bounded, there is a point-\weakstar{} convergent subnet of $\kappa \circ \phi_i$.
  For ease of notation we drop $\kappa$ and consider $\phi_i$ as maps $\Ind I \rightarrow M$,
  so that passing to a subnet 
  we may assume that $\phi_i(f)$ converges to $\phi(f)$ in the \weakstar-topology
  for some map $\phi\colon \Ind I \rightarrow M$ and every $f\in \Ind I$.
  The limit $\phi$ will still be a $C(X)$-linear ucp map fixing $C(X)$,
  as all of these are pointwise conditions, 
  but we need to check that it is a $\G$-map.
  That is, for every $f,g \in \Ind I$ and $\gamma \in \G^{v}_{u}$ 
  with $\gamma. [f]_u = [g]_v$,
  we need to show that 
  \begin{equation}
    \gamma. [\phi(f)]_u = [\phi(g)]_v.
    \label{eq:weaklyContinuousAction}
  \end{equation}
  Fixing an open bisection $B$ containing $\gamma$, 
  we may assume that $f$ is supported in $r^{-1}( s\left( B \right) )$ 
  and that $g = B.f$.

  We claim that $\phi(B.f) = B.\phi(f)$:
  The predual of $M$,
  \[
    M_\ast 
    \cong
    \ell^1 \left( \left\{ 
        \left( \gamma, \tilde x \right) \in \G \times \tilde X
        \bigm|
        s(\gamma) = q(\tilde x)
    \right\} \right)
    \cong 
    \ell^{1}\left( \G, 
      \G \ast_s \bigsqcup_{x \in X}\ell^{1}\left( q^{-1}\left( x \right) \right) 
    \right),
  \]
  carries an analogous $\G$-action.
  For $f\in \Ind I \subseteq M$ and $\chi \in M_\ast$, 
  the evaluation $f(\chi)$ only depends on the values of $\chi$ on the support of $f$,
  seen as a function on $\Gzero$.
  Therefore, when $f$ is supported in $s(B)$, 
  $B.f$ is supported in $r(B)$.
  To evaluate $\left( B.f \right)(\chi)$ we may assume that $\chi$ is supported on $r(B)$ 
  and find
  $\left( B.f \right)\left( \chi \right) = f\left( B^{-1}(\chi) \right)$.
  Now for $\supp(f) \subseteq s(B)$ and $\supp(\chi)\subseteq r(B)$ we calculate
  \begin{align*}
    \phi_i(B.f)(\chi)
    = \left( B.\phi_i(f) \right)(\chi)
    = \phi_i(f)\left( B^{-1}.\chi \right)
    \rightarrow
    \phi(f)\left( B^{-1}.\chi \right)
    = \left( B.(\phi(f)) \right)(\chi).
  \end{align*}
  As the left hand expression converges to $\phi(B.f)(\chi)$, 
  we can conclude that $\phi(B.f) = B.\phi(f)$.
  Since $B.f$ is a valid choice of $\gamma$ in Equation \eqref{eq:weaklyContinuousAction},
  we conclude that
  \begin{align*}
    \gamma.\left( \phi_u\left( [f]_u \right) \right)
    = \gamma.[\phi(f)]_u  
    = [B.\phi(f)]_v 
    = [\phi(B.f)]_v
    = \phi_v\left( [B.f]_v \right)
    = \phi_v\left( \gamma. [f]_u \right).
  \end{align*}
  As every element of $\left( \Ind I \right)_u$ is of the form $[f]_u$ for some such $f$,
  this proves $\G$-equivariance of the limit $\phi$.

  We may thus define a $C(X)$-seminorm $p$ on $\Ind I$ by $p(f) = \|E\circ\phi(f)\|_{\Ind I}$.
  As in \cite[Lemma 3.4]{HamanaOperatorSystems}, we then find that
  \begin{equation*}
    p(f) = \|E \circ \phi(f)\|
    \leq \|\phi(f)\|
    \leq \limsup \|\phi_j(f)\|
    = \lim p_i(f),
  \end{equation*}
  with $j$ indexing the convergent subnet chosen before.
  Now $p$ is the desired lower bound.
\end{proof}

By Zorn's lemma we know of the existence of a minimal $C(X)$-seminorm on $\Ind I$,
from which we obtain a minimal $C(X)$-projection:

\begin{lemma}
  There is a minimal $C(X)$-projection on $\Ind I$.
  \label{lem:minimalProjection}
\end{lemma}

\begin{proof}
  We follow \cite[Thm 3.5]{HamanaOperatorSystems}, 
  which originally followed \cite[Thm 1]{KaufmanByHamana}.
  Let $p$ be a minimal $C\left( X \right)$-seminorm and $\phi$ the ucp map implementing it.
  We show that $\phi$ is a projection.

  Define the net
  \[\phi_n \colon \Ind I \rightarrow \Ind I \subseteq M
    \qquad \text{by} \qquad
    \phi_n \defeq \left( \phi + \phi^2 + \cdots + \phi^n \right)/n,
  \]
  and pass to a point-\weakstar{} convergent subnet
  as in the proof of Lemma \ref{lem:lowerBound},
  and denote the limit by $\phi_\infty\colon \Ind I \rightarrow M$.
  Using the conditional expectation $E\colon M\rightarrow \Ind I$
  as before, we note that $E\circ \phi_\infty$ induces a $C(X)$-seminorm
  and for $f\in \Ind I$ we obtain that
  \begin{align*}
    \|E\circ \phi_\infty(f)\|
    \leq \|\phi_\infty(f)\|
    \leq \limsup \|\phi_n(f)\|
    \leq \|\phi(f)\|
    = p(f),
  \end{align*}
  which implies $\|\phi(f)\| = \limsup \|\phi_n(f)\|$
  by minimality of $p$.
  Hence
  \[
    \|\phi(f) - \phi^2(f)\|
    = \| \phi\left( f - \phi(f) \right)\|
    = \limsup \|\phi_n\left( f-\phi(f) \right)\|
    = 0,
  \]
  so $\phi$ is idempotent and therefore a $C(X)$-projection.

  Among $C(X)$-projections, $\phi$ is also minimal:
  Given any other $C(X)$-projection $\psi$
  with $\psi \leq \phi$, or equivalently $\psi\circ\phi = \psi = \phi\circ\psi$,
  we find that
  \[
    \|\psi(f)\|
    = \| \psi\circ \phi(f)\|
    \leq \|\phi(f)\|
    = p(f),
  \]
  and so $\psi$ and $\phi$ define the same seminorm by minimality of $p$.
  In particular, $\ker\psi = \ker \phi$.
  As $\psi$ is idempotent, 
  we have $\psi(f) - f \in \ker\psi = \ker \phi$,
  for every $f\in \Ind I$,
  and
  $\psi(f) = \phi\circ\psi(f) = \phi(f)$,
  so the two projections coincide on all of $\Ind I$.
\end{proof}

From a minimal $C(X)$-projection $\phi$, we build the Choi--Effros algebra $\Cs(\phi)$,
as originally constructed in \cite[Theorem 3.1]{ChoiEffros} 
and described in the context of injective envelopes in \cite[Theorem 2.3]{HamanaCstar}.
This is done by equipping its range with the Choi--Effros product defined below,
turning the range of $\phi$ into a \Cs-algebra 
that is completely order isomorphic to the range of $\phi$ as an operator system.
In our setting, we have to verify that $\Cs(\phi)$ is indeed a \Galgebra{},
and that the ucp map $\Ind I \rightarrow \Cs(\phi)$ induced by $\phi$ is $\G$-equivariant.
\begin{proposition}
  For a $C(X)$-projection $\phi$ on $\Ind I$, 
  the Choi--Effros algebra $\Cs(\phi)$ can be given the structure of a \Galgebra{}
  such that the ucp map $\Ind I \rightarrow \Cs(\phi)$ induced by $\phi$ 
  is $\G$-equivariant.
  \label{prop:ChoiEffros}
\end{proposition}
\begin{proof}
  First we check that $C(X)$ still acts by central multipliers on $\Cs(\phi)$, 
  giving it the structure of a $C(X)$-algebra.
  The algebra $\Cs(\phi)$ is the range of $\phi$ inside $\Ind I$ as underlying set,
  equipped with the norm of $\Ind I$ and multiplication ${\circ}$ given by $x \circ y = \phi(x y)$.
  As $\phi$ is a $C(X)$-map, 
  its range is closed under multiplication by $g \in C(X)$.
  Inheriting this action still gives an adjointable map that is its own adjoint, as
  \[
    x \circ \left( f.y \right)
    = \phi\left( x\left( f.y \right) \right)
    = \phi\left( \left( f.x \right)y \right)
    = \left( f.x \right) \circ y,
  \]
  for $xmy \in \Ind I$ and $f\in C(\tilde X)$.
  We conclude that the action of $C(\tilde X)$ on $\Ind I$
  via $y \mapsto f.y$ is by central multipliers.
  Hence, the surjective ucp map 
  $j\colon \Ind I \rightarrow \Cs(\phi)$ 
  induced by $\phi$ as $j(x) = \phi(x)$
  factors through the fibres to obtain maps 
  $j_u\colon \left( \Ind I \right)_u \rightarrow \left( \Cs(\phi) \right)_u$.
  For $j$ to be $\G$-equivariant, the $\G$-action on $\Cs(\phi)$ has to be given by
  \[
    \gamma.\left( j_u([f]_u) \right) = j_u\left( \gamma.[f]_u \right) 
    \qquad
    \text{for $f\in \Ind I$},
  \]
  where the action of $\gamma$ is the $\G$-action associated with the appropriate \Csalgebra{}.
  Picking an open bisection $B$ around $\gamma$ 
  and cutting down the section $f$ to be supported in $s(B)$ as above yields
  \[
    \gamma.\left( j_u([f]_u) \right) 
    = \left[ j\left( B.f \right) \right]_u 
    = \left[ B.j(f) \right]_v,
  \]
  where the action of $B$ on appropriately supported functions in $\Cs(\phi)$ 
  is as a subspace of $\Ind I$.
  Continuity of the action is shown 
  just as in Equation \eqref{eq:continuousActionByBisections}.
\end{proof}

The following arguments need almost no modification to those given in the work of Hamana.
First, we show that $\Cs(\phi)$ is a $\G$-rigid extension 
by noting that it carries a unique $C(X)$-seminorm
and repeating the arguments of Lemma \ref{lem:minimalProjection}.
Then we show that every rigid injective extension is also essential,
making it an injective envelope.
\begin{lemma}
  For $\phi$ a minimal $C(X)$-projection on $\Ind I$,
  the extension $C(X) \hookrightarrow \Cs(\phi)$ 
  is a $\G$-rigid, $\G$-injective $\G$-extension.
  \label{lem:rigidInjectiveExtension}
\end{lemma}
\begin{proof}
  By \cite[Thm 3.1]{ChoiEffros}, $\Cs(\phi)$ is injective 
  and even $\G$-injective, as $\phi$ is a $\G$-map.

  From any $C(X)$-seminorm on $\Cs(\phi)$ given by a ucp $\G$-map $\phi'$
  we obtain a $C(X)$-seminorm on $\Ind I$ given by $\phi'\circ\phi$.
  As it is then dominated by the minimal $C(X)$-seminorm of $\phi$,
  the \Cs-norm on $\Cs(\phi)$ is its unique $C(X)$-seminorm.
  Let $\psi\colon \Cs(\phi) \rightarrow \Cs(\phi)$ be
  a ucp $\G$-map restricting to the identity on $C(X)$.
  Therefore, a minimal $C(X)$-seminorm on $\Cs(\phi)$
  is induced by the identity
  and, analogously to the proof of Lemma \ref{lem:minimalProjection},
  for $a\in \Cs(\phi)$ we obtain that
  \[
    \limsup\left\| \left( \psi(a) + \psi^2(a) + \ldots + \psi^n(a) \right)/n \right\|
    =
    \limsup\|\psi_n(a)\| = \|a\|
  \]
  and hence 
  \[
    \|\psi(a) - a\|
    = \limsup \| \psi_n\left( \psi(a) - a \right)\|
    = 0.
  \]
  Therefore $\psi$ is the identity.
\end{proof}
\begin{lemma}
  Every $\G$-injective $\G$-rigid extension is also $\G$-essential.
  \label{lem:rigidInjectiveImpliesEssential}
\end{lemma}

\begin{proof}
  The proof is almost abstract nonsense as in \cite[Lemma 3.7]{HamanaOperatorSystems}.

  Let $\left( I, \iota \right)$ be a $\G$-injective $\G$-extension of $A$
  and $\phi\colon I \rightarrow B$ a ucp $\G$-map 
  such that $\phi\circ \iota\colon A\rightarrow B$ is injective.
  We have to show that $\phi$ itself is already injective.
  By $\G$-injectivity of $I$
  we find a ucp $\G$-map $\psi\colon B \rightarrow I$
  that restricts to $\iota$ when $A$ is seen as a subalgebra of $B$
  via the embedding $\phi \circ \iota$, as in Diagram \eqref{cd:injectiveRigidImpliesEssential}:
  \begin{equation}
    \begin{tikzcd}
      B
      \ar[r, "\psi"above, dashed]
      & I
      \\
      A
      \ar[u, "\phi \circ \iota"left, hookrightarrow]
      \ar[ur, "\iota"below right]
    \end{tikzcd}
    \label{cd:injectiveRigidImpliesEssential}
  \end{equation}
  Now $\psi\circ \phi$ is a ucp $\G$-selfmap of $I$
  that restricts to the identity on $\iota(A)$. 
  By $\G$-rigidity it is therefore the identity,
  implying that $\phi$ is injective.
\end{proof}
Finally, we may remark that this injective envelope is unique 
up to $\G$-equivariant isomorphism preserving the embedding:
\begin{theorem}
  For an \etale{} groupoid $\G$ with compact unit space $X$
  the \Galgebra{} $C\left( X \right)$ has a $\G$-equivariant injective envelope 
  $I_{\G}\left( C\left( X \right) \right)$,
  such that for any other $\G$-equivariant injective envelope $Z$ 
  there is a unique $\G$-isomorphism 
$\psi\colon I_{\G}\left( C\left( X \right) \right) \rightarrow Z$ for which
  \begin{equation}\label{cd:uniqueInjectiveEnvelope}
    \begin{tikzcd}
      I_{\G}\left( C\left( X \right) \right)
      \ar[r, rightarrow, "\psi" above]
      &Z\\
      C\left( X \right)
      \ar[u, hookrightarrow]
      \ar[ur, hookrightarrow]
    \end{tikzcd}
  \end{equation}
  commutes.
  We call the spectrum of $I_\G(C(X))$ the Furstenberg boundary of $\G$.
  \label{thm:uniqueInjectiveEnvelope}
\end{theorem}
  This works as for Hamana in \cite[Lemma 3.8]{HamanaCstar}:
\begin{proof}
  Let $I_\G(C(X))$ be the $\G$-injective $\G$-rigid extension above and
  let $Z$ be another $\G$-injective $\G$-essential ectension.
  We may extend the inclusions of $C\left( X \right)$ into one of the envelopes
  to the other envelope as in Diagram \eqref{cd:uniqueEnvelope}
  by $\G$-injectivity:
  \begin{equation}
    \begin{tikzcd}
      I_\G\left( C\left( X \right) \right)
      \ar[r, dashed, "\psi"above]     
      & Z
      \\
      C\left( X \right)
      \ar[u, hookrightarrow]
      \ar[ur, hookrightarrow]
    \end{tikzcd}
    \qquad
    \begin{tikzcd}
      I_\G\left( C\left( X \right) \right)
      &Z
      \ar[l, dashed, "\hat\psi"above]
      \\
      &C\left( X \right)
      \ar[u, hookrightarrow]
      \ar[ul, hookrightarrow]
    \end{tikzcd}
    \label{cd:uniqueEnvelope}
  \end{equation}
  As $\hat\psi\circ\psi$ restricts to the identity on $C\left( X \right)$,
  it is the identity on $I_\G(C(X))$ by $\G$-rigidity, 
  and in particulat $\hat \psi$ is surjective.
  Furthermore, $\hat \psi$ is injective by $\G$-essentiality of $Z$  
  and hence a $\G$-isomorphism.
\end{proof}

\section{Some Properties of Groupoid Boundaries}\label{sec:properties}
Consider a locally compact Hausdorff \etale{} groupoid $\G$
with compact unit space $X$ as above.
Denote by $\tilde X$ its Furstenberg boundary,
the spectrum of the $\G$-equivariant injective envelope $I_\G\left( C(X) \right)$
of $C(X)$.
Passing to the crossed product groupoid $\tilde \G\defeq \tilde X \rtimes \G$
associated with the action of $\G$ on $C(\tilde X)$,
we treat the boundary as a second, larger groupoid $\tilde \G$
that contains $\G$ as a quotient.
We call $\tilde \G$ the \emph{boundary groupoid} of $\G$.

We explore some properties that make $\tilde \G$ more tractable.
The following propositions are generalisations of \cite[Proposition 1.3]{Kawabe}.
\begin{proposition}
  All stabiliser groups of the boundary groupoid $\tilde \G$ 
  of a Hausdorff \etale{} groupoid $\G$ are amenable.
  \label{prop:amenableStabilisers}
\end{proposition}
\begin{proof}
  Consider the \Galgebra{} 
  $L = \left\{ f \in \ell^\infty(\G)
  \mid x \mapsto \| f|_{\G^x} \|_\infty 
  \text{ is upper-semicontinuous}\right\}$ 
  with left action of $\G$ on the argument.
  Identifying $C(X)$ as a unital sub-\Galgebra{}, 
  we find a ucp $\G$-map $\phi \colon L \rightarrow C(\tilde X)$ by $\G$-injectivity.
  Passing to the fibre at any $x \in X$, 
  we obtain a $\G_x^x$-equivariant ucp map
  $\phi_x \colon \ell^\infty(\G^x) \rightarrow C(q^{-1}(x))$.
  For any $\tilde x \in q^{-1}(x)$, we may identify the stabiliser group
  $\tilde\G_{\tilde x}^{\tilde x}$ of $\tilde \G$ at $\tilde x$
  with a subgroup of $\G_x^x$
  and obtain a $\tilde\G_{\tilde x}^{\tilde x}$-equivariant unital
  \starhomomorphism{}
  $\ell^\infty(\tilde \G_{\tilde x}^{\tilde x}) \rightarrow
  \ell^\infty(\G^x)$
  since each orbit of $_{\tilde\G_{\tilde x}^{\tilde x}} \backslash^{\G^x}$ is in bijection
  with $\tilde\G_{\tilde x}^{\tilde x}$.
  Then, composition with the evaluation $\text{ev}_{\tilde x}$ at $\tilde x$, 
  \[
    \ell^\infty(\tilde\G_{\tilde x}^{\tilde x})
    \rightarrow
    \ell^\infty(\G^x)
    \xrightarrow{\phi_x}
    C(q^{-1}(x))
    \xrightarrow{\text{ev}_{\tilde x}}
    \Cpx
  \]
  gives a state on $ \ell^\infty(\tilde\G_{\tilde x}^{\tilde x})$ that is
  invariant under the left action of $\tilde\G_{\tilde x}^{\tilde x}$.
\end{proof}

Let $q\colon \tilde X \rightarrow X$
be the quotient map
obtained from the embedding $C(X) \hookrightarrow C(\tilde X)$.
\begin{proposition}
  Let $\tilde X$ be the Furstenberg boundary of a groupoid $\G$ with compact unit space $X$
  and $q\colon \tilde X \rightarrow X$ the continuous surjection
  obtained from the inclusion $C(X) \hookrightarrow C(\tilde X)$.
  Let $Z\subseteq \tilde X$ be a closed $\G$-invariant subset such that $q(Z) = X$. 
  Then $Z = \tilde X$.
  \label{prop:equivariantCover}
\end{proposition}
\begin{proof}
  As $q$ restricted to $Z$ is still surjective, we can embed $C(X)$ into $C(Z)$ 
  and, by $\G$-injectivity of $C(\tilde X)$,
  extend the embedding $C(X) \hookrightarrow C(\tilde X)$ 
  to a ucp map $\phi\colon C(Z) \rightarrow C(\tilde X)$.
  Then the composition
  with restriction to $Z$ 
  \[ C(\tilde X) \xrightarrow{\text{res}_Z} C(Z) \xrightarrow \phi C(\tilde X) \]
  is a $\G$-equivariant ucp map $C(\tilde X) \rightarrow C(\tilde X)$ 
  that restricts to the identity on the subalgebra $C(X)$.
  By rigidity of the $\G$-equivariant injective envelope, 
  it is therefore the identity on all of $C(\tilde X)$.
  Hence, $\text{res}_Z$ is injective and $Z$ is dense. 
  As it is also closed, we get $Z=\tilde X$.
\end{proof}

\begin{proposition}
  Any boundary groupoid $\tilde \G$ is ample.
  \label{prop:ample}
\end{proposition}
\begin{proof}
  As $C(\tilde X)$ is an injective \Csalgebra{}, $\tilde X$ is Stonean
  and in particular zero-dimensional.
  Any \etale{} groupoid with totally disconnected unit space is ample.
\end{proof}

\begin{proposition}
  If every orbit of a boundary groupoid $\tilde \G$ has at least two points,
  then the isotropy bundle $\text{Iso}( \tilde \G )$ of $\tilde \G$ is clopen.
  \label{prop:clopenIsotropy}
\end{proposition}
If $\G$ already has orbits consisting of at least two points, 
then so does $\tilde \G$.
We are of course tacitly avoiding the case where $\G$ is a group
and $\Gzero$ is a single point.
\begin{proof}
  For every groupoid the isotropy bundle is closed, so we have to show that it is open.
  Consider $\gamma \in \tilde \G_{\tilde x}^{\tilde x}$.
  As $\tilde \G$ is ample, there exists a \textit{full} open bisection $B$ containing $\gamma$,
  that is,
  an open bisection with $s(B) = \tilde X = r(B)$
  by Brix--Scarparo \cite{Kevin}.
  Then $B$ defines a homeomorphism $\tilde X \rightarrow \tilde X$
  by $r\circ (s|_{B'})^{-1}$
  whose fixed point set $F$ contains $s(\gamma) = \tilde x = r(\gamma)$.
  By Frolik's theorem \cite{Frolik}
  this fixed point set is open
  and $B \cap s^{-1}(F)$ is an open neighbourhood of $\gamma$
  contained in the isotropy.
\end{proof}
\begin{remark}
  This of course works in general for every groupoid with Stonean unit space
  and orbits consisting of at least two points.
  So far we have been unable to drop this second condition, 
  although it can be dropped for the case of crossed products by groups,
  where the existence of enough full bisections
  for the topological full group to be covering is trivial
  (see \cite{Kevin} for definitions).
  Note that it is however not necessary for every $\gamma$ to be contained
  in a full bisection in order to apply Frolik's theorem as above.
  It suffices that $\gamma$ is contained in an open bisection $B$
  whose source $s(B)$ is again a Stonean space that contains the range $r(B)$.
\end{remark}

Furthermore, we show that the notion of groupoid Furstenberg boundary 
generalises the boundaries of groups
and even groups acting on spaces.
\begin{proposition}
  Let $\G = X \rtimes \Gamma$ be the transformation groupoid
  of a discrete group $\Gamma$ acting on a compact Hausdorff space $X$.
  Then the groupoid-equivariant injective envelope $I_\G(C(X))$
  coincides with the group-equivariant injective envelope $I_\Gamma(C(X))$.

  Hence, the groupoid Furstenberg boundary of a discrete group is its Furstenberg boundary.
  \label{prop:generalisesKawabe}
\end{proposition}
\begin{proof}
  First observe that $I_\G \defeq I_\G(C(X))$ carries a $\Gamma$-action 
  where $\gamma.f(\tilde x) = f(r(\tilde x,\gamma))$,
  with $\gamma\in \Gamma$, $f\in I_\G$, and $\tilde x$ in the spectrum of $I_\G$.
  The action is continuous by continuity of the $\G$-action.
  Likewise, $I_\Gamma \defeq I_\Gamma(C(X))$ carries a $\G$-action:
  As $C(X)$ embeds into the commutative \Cs-algebra $I_\Gamma$, it is a $C(X)$-algebra
  and as before we may define a continuous action by 
  $(x,\gamma)[f]_x \defeq [\gamma.f]_{\gamma.x}$
  with $\gamma \in \Gamma$, $f\in I_\Gamma$, and $x\in X$.
  
  Using their injectivity in the respective category,
  we may therefore extend the embeddings of $C(X)$ into $I_\Gamma$ and $I_\G$
  to a ucp $\Gamma$-map $\phi\colon I_\G \rightarrow I_\Gamma$
  and a ucp $\G$-map $\psi\colon I_\Gamma \rightarrow I_\G$.
  We then calculate that for all $x \in X$
  \begin{align*}
    [\psi(\gamma.f)]_{\gamma.x}
    &= \psi_{\gamma.x}( (x,\gamma).[f]_{x})
    =  (x, \gamma).\psi_{x}([f]_{x})
    =  [\gamma.\psi(f)]_{\gamma.x},
  \end{align*}
  so $\psi$ is also a ucp $\Gamma$-map.
  On the other hand, $\phi$ contains $C(X)$ in its multiplicative domain 
  and is therefore a $C(X)$-map.
  Hence for all $x\in X$ we find that
  \begin{align*}
    \phi_{\gamma.x}\left( (x,\gamma).[f]_x \right)
    = \phi_{\gamma.x}\left( [\gamma.f]_{\gamma.x} \right)
    = \left[ \phi(\gamma.f) \right]_{\gamma.x}
    = \left[ \gamma.\phi(f) \right]_{\gamma.x}
    = (x, \gamma).\phi_x([f]_x),
  \end{align*}
  so $\phi$ is $\G$-equivariant.

  Now $\psi \circ \phi$ is a ucp $\G$-map $I_\G \rightarrow I_\G$ that fixes $C(X)$
  and is therefore the identity by $\G$-rigidity of $I_\G$.
  Likewise, $\phi \circ \psi$ is a ucp $\Gamma$-map $I_\Gamma \rightarrow I_\Gamma$
  that fixes $C(X)$ and is the identity by $\Gamma$-rigidity of $I_\Gamma$
  and the two envelopes are isomorphic \Cs-algebras.
\end{proof}

\section[Simplicity]{Simplicity of Groupoid \Csalgebra{}s}\label{sec:simple}
\label{sec:simplicity}
In this final section, we apply the theory of boundary groupoids
to obtain a sufficient criterion for a groupoid $\G$ to have the intersection
property in theorem \ref{thm:conditionSufficient}.
In the case of a minimal groupoid, this
provides a weaker sufficient criterion for \Cs{}-simplicity
than the widely used notion of \emph{topological principality}
(sometimes also called \emph{topological freeness}).

We first relate the \Csalgebra{}s of a groupoid and its boundary groupoid:
\begin{lemma}
  For a locally compact Hausdorff \etale{} groupoid $\G$ 
  with compact unit space
  there is a canonical embedding of $\Csr(\G)$ 
  into the reduced \Csalgebra{} $\Csr(\tilde \G)$ of its boundary groupoid $\tilde \G$.
  \label{lem:reducedAlgebraEmbeds}
\end{lemma}
\begin{proof}
  Denote again the quotient map $\tilde X \rightarrow X$ by $q$.
  As a set, $\tilde \G$ is 
  $\left\{ (\gamma, \tilde x) \bigm | \gamma \in \G, \tilde x \in \tilde X, r(\gamma) 
  = q(\tilde x) \right\}$
  and the range and source maps are given by 
  $r\left( (\gamma, \tilde x) \right) = \tilde x$
  and
  $s\left( (\gamma, \tilde x) \right) = \gamma^{-1}.\tilde x$.
  Furthermore denote by $Q\colon \tilde \G \rightarrow \G$ the surjective groupoid homomorphism 
  given by
  $(\gamma, \tilde x) \mapsto \gamma$, which extends $q\colon \tilde X \rightarrow X$.
  Note that $Q$ is proper so that
  $C_{\text c}(\G)$ embeds into $C_{\text c}(\tilde \G)$.
  Indeed, if $K\subset \G$ is a compact set,
  then the preimage $Q^{-1}(K) = \{ (\gamma, \tilde x) \mid \gamma \in K\}$
  is compact, since $\tilde X$ is.

  Since $Q$ is a groupoid homomorphism, 
  precomposition with $Q$ as a map
  $Q^\ast\colon C_{\text c}(\G) \hookrightarrow C_{\text c}(\tilde \G)$ 
  is compatible with the convolution and involution on both algebras.
  It is furthermore isometric with respect to the reduced norms,
  $\|f\|_{\Csr(\G)} = \|f\circ Q\|_{\Csr(\tilde \G)}$
  for $f \in C_c(\G)$,
  so that $Q^\ast$ extends from $C_{\text c}(\G)$ 
  to an embedding of the associated reduced algebras:
  First note that $Q$ gives a bijection 
  $\tilde \G_{\tilde x} \rightarrow \G_{q(\tilde x)}$ 
  between the source fibres by $(\gamma, \gamma\tilde x) \mapsto \gamma$,
  which makes
  $U\colon \ell^{2}(\tilde \G_{\tilde x}) \rightarrow \ell^{2}(\G_{q(\tilde x)})$ 
  by $\delta_{(\gamma, \gamma\tilde x)} \mapsto \delta_\gamma$,
  an isomorphism that
  intertwines the associated source fibre representations $\pi_{\tilde x}$ and $\pi_x$:
  \begin{align*}
    U \pi_{\tilde x}(f\circ Q) U^{\ast} \,\delta_\gamma
    &= U \pi_{\tilde x}(f\circ Q) \,\delta_{(\gamma, \gamma \tilde x)}\\
    &= U \sum\limits_{\alpha \in \G_x}
    f\circ Q\left( (\alpha, \alpha\tilde x)\right)
    \,\delta_{(\alpha\gamma, \alpha\gamma\tilde x)}\\
    &= U \sum\limits_{\alpha \in \G_{x}}
    f( \alpha ) \,\delta_{(\alpha\gamma, \alpha\gamma\tilde x)}\\
    &= \sum\limits_{\alpha \in \G_{x}} f( \alpha ) \,\delta_{\alpha\gamma}\\
    &= \pi_x(f) \,\delta_\gamma.
  \end{align*}
  Hence the norms of $f$ and $f\circ Q$ coincide 
  and we obtain an embedding of $\Csr(\G)$ into $\Csr(\tilde \G)$.
\end{proof}

To relate the intersection properties of $\G$ and $\tilde \G$, we furthermore
need the following technical lemma
stating that the above embedding is compatible with the $\G$-action
under ucp extensions:
\begin{lemma}
  Let $\G$ be a groupoid as above and $\tilde \G$ its boundary groupoid.
  Let $\pi$ be a $\ast$-re\-pre\-sen\-tation of $\Csr(\G)$ on $\B(\H)$
  whose restriction to $C(X)$ is injective,
  and let $\tilde\pi\colon \Csr( \tilde \G ) \rightarrow \B(\H)$
  be a ucp extension of $\pi$.
  Then the \Csalgebra{} $\intermediaryAlgebra$ 
  generated by $\tilde\pi( C( \tilde X ) )$
  inside $\B(\H)$
  carries the structure of a \Galgebra{},
  such that the restriction
  $\tilde\pi|_{C( \tilde X )} \rightarrow E$ 
  is $\G$-equivariant.
  \label{lem:extendUcpToGAlgebra}
\end{lemma}
\begin{proof}
  Let
  $\intermediaryAlgebra \defeq C^{\ast}( \tilde \pi( C(\tilde X) ) ) 
  \subseteq \B(\H)$ 
  be the $\Cs$-algebra generated in $\B(\H)$
  by the image of $\tilde\pi$ on $C( \tilde X )$.
  As $\tilde \pi$ restricts to $\pi$ on $C(X)$ and is injective there,
  we may identify 
  $C(X) \subseteq \intermediaryAlgebra$ 
  as a subalgebra and since $\pi$ is a \starhomomorphism{}, 
  $C( X )$ lies in the multiplicative domain of $\tilde \pi$, 
  so that its action is central.
  Therefore $\intermediaryAlgebra$, being unital, is a $C(X)$-algebra.
  For $\tilde \pi$ to be $\G$-equivariant, 
  the $\G$-action on 
  $\tilde\pi( C( \tilde X ) ) \subseteq \intermediaryAlgebra$
  is determined by
  \[\gamma. \left[ \tilde \pi(f) \right]_x
    =
    \tilde \pi_{\gamma.x}\left( \gamma.\left[ f \right]_x \right)
  \]
  for $f\in C( \tilde X )$ with $x = s(\gamma)$
  and $\gamma.x = r(\gamma)$.
  Note that if 
  $g_\gamma \in C_c(\G)$ is supported on a bisection of $\G$ 
  and $g_\gamma(\gamma) = 1$,
  then $g_\gamma \ast f \ast g_\gamma^\ast$ is in $C( \tilde X )$
  and a quick calculation shows that 
  \begin{equation}
    \label{eq:conjugationActionExplicit}
    \left( g_\gamma \ast f \ast g_\gamma^\ast \right)(\tilde x) 
    =
    \left| g_\gamma\left( \left(r|_B \right)^{-1}(\tilde x)\right) \right|^2 
    \cdot f\left(s\circ\left( r|_B \right)^{-1}(\tilde x)\right),
  \end{equation}
  where $B\subseteq \tilde \G$ is the bisection
  on which $g_\gamma\circ q$ is supported.
  Hence $\gamma.[f]_x = \left[ g_\gamma \ast f \ast g_\gamma^\ast \right]_{\gamma.x}$
  and since $g_\gamma$ is in the multiplicative domain of $\tilde \pi$
  we can define
  \[
    \gamma. \left[ \tilde\pi(f) \right]_x
    \defeq \tilde\pi_{\gamma.x}\left( \gamma.\left[ f \right]_x \right)
    = \left[ \tilde\pi(g_\gamma) \tilde\pi(f) \tilde\pi(g_\gamma)^\ast \right]_{\gamma.x}.
  \]
  This extends to general $a \in \intermediaryAlgebra$ by
  \begin{equation}
    \gamma. \left[ a \right]_x
    \defeq
    \left[ \tilde\pi(g_\gamma) a \tilde\pi(g_\gamma)^\ast \right]_{\gamma.x}
    \label{eq:actionOnIntermediaryAlgebra}
  \end{equation}
  and we proceed to show that this is a well-defined $\G$-action.
  First we show that $\gamma.[a]_x$ does not depend on the choice of $a$ to
  represent $[a]_x$,
  that is, $\gamma.[a]_x = 0$ if $[a]_x=0$
  for $a \in \overline{C_0( X\setminus x )\intermediaryAlgebra}$.
  We may approximate $a$ by finite sums of elements of the form 
  $h \cdot \tilde\pi(f_1)\cdots\tilde\pi(f_n)$
  with $h \in C_0( X\setminus x )$
  and $f_i \in C( \tilde X )$
  while $n$ varies
  and as the action in Equation \eqref{eq:actionOnIntermediaryAlgebra}
  depends continuously on $a$,
  it suffices to show that 
  $\gamma.\left[ h \cdot \tilde\pi(f_1) \cdots \tilde\pi(f_n) \right]_x = 0$.
  As $g_\gamma^\ast \ast g_\gamma$ and $g_\gamma \ast g_\gamma^\ast$
  are supported on $X$
  and $\left[ g_\gamma \ast g_\gamma^\ast \right]_{\gamma.x} = \1$,
  we may now calculate that
  \begin{align*}
    \left[ \pi\left(g_\gamma\right) h \tilde\pi(f_1) 
    \tilde\pi(f_2) \pi\left( g_\gamma \right)^\ast \right]_x
    &=
    \left[ \pi\left( g_\gamma \ast g_\gamma^\ast \right) \right]_{\gamma.x}
    \left[ \pi\left(g_\gamma\right) h \tilde\pi(f_1) 
    \tilde\pi(f_2) \pi\left( g_\gamma \right)^\ast \right]_x
    \\
    &=
    \left[ \pi\left( g_\gamma \ast g_\gamma^\ast \right) 
      \pi\left(g_\gamma\right) h \tilde\pi(f_1) 
    \tilde\pi(f_2) \pi\left( g_\gamma \right)^\ast \right]_x
    \\
    &=
    \left[ \pi\left( g_\gamma \right) 
      \pi\left( g_\gamma^\ast  \ast g_\gamma\right) 
      h \tilde\pi(f_1) 
    \tilde\pi(f_2) \pi\left( g_\gamma \right)^\ast \right]_x
    \\
    &=
    \left[ \pi\left( g_\gamma \right) 
      h \tilde\pi(f_1) 
      \pi\left( g_\gamma\right)^\ast  \pi\left(g_\gamma\right) 
    \tilde\pi(f_2) \pi\left( g_\gamma \right)^\ast \right]_x
    \\
    &=
    \left[ 
      \pi\left( g_\gamma \right) 
      h \tilde\pi(f_1) 
      \pi\left( g_\gamma\right)^\ast
    \right]_{\gamma.x}
    \left[ 
      \pi\left(g_\gamma\right) 
      \tilde\pi(f_2) \pi\left( g_\gamma \right)^\ast 
    \right]_x
    \\
    &=
    \tilde\pi_{\gamma.x}\left( \gamma.\left[ hf_1 \right]_x \right)
    \tilde\pi_{\gamma.x}\left( \gamma.\left[ f_2 \right]_x \right)
    = 0
  \end{align*}
  for $n=2$ and analogously for arbitrary $n$.
  Incidentally, this also shows how the action gives homomorphisms
  between the appropriate fibers, that is,
  \[ \gamma.\left[ a b \right]_x =
    \gamma.\left[ a \right]_x
    \gamma.\left[ b \right]_x.
  \]

  Next we show that $\gamma.\left[ a \right]_x$ is independent 
  of the choice of $g_\gamma$ as above.
  Suppose $g_\gamma$ and $g'_\gamma$ are as described 
  supported on open bisections $B$ and $B'$
  and evaluate to one at $\gamma$.
  As $B\cap B'$ is an open neighbourhood of $\gamma$,
  the partial homeomorphisms
  $s\circ \left( r|_B \right)^{-1}$ and
  $s\circ \left( r|_{B'} \right)^{-1}$
  coincide on the neighbourhood $r(B\cap B')$ of $r(\gamma)$
  and by Equation \eqref{eq:conjugationActionExplicit}
  we find that
  \[
    \left(g_\gamma \ast f \ast g_\gamma^\ast\right)
    - 
    \left(g'_\gamma \ast f \ast \left( g'_\gamma \right)^{\ast}\right)
    \left( \tilde x \right)
    =
    \underbrace{
      \left( 
        \left| g_\gamma\left( \left( r|_B \right)^{-1}(x) \right) \right|^2
        - 
        \left| g'_\gamma\left( \left( r|_B \right)^{-1}(x) \right) \right|^2
      \right)
    }_{\in C_0\left( X\setminus r(\gamma) \right)}
    \cdot f\left(s\circ\left( r|_B \right)^{-1}(\tilde x)\right).
  \]
  and therefore
  \[
    \left[ g_\gamma \ast f \ast g_\gamma^\ast \right]_{r(\gamma)}
    =
    \left[ g'_\gamma \ast f \ast \left( g'_\gamma \right)^\ast \right]_{r(\gamma)}.
  \]
  By approximating arbitrary $a \in \intermediaryAlgebra$
  by finite sums of elements of the form $h\cdot \tilde\pi(f_1) \cdots \tilde \pi(f_n)$
  as above,
  we conclude that the same holds for arbitrary $a$.

  As $\left[ g_\gamma^\ast \ast g_\gamma \right]_x = \1$
  and $g_\gamma^\ast$ is a valid choice for $g_{\gamma^{-1}}$,
  the action of $\gamma$ is invertible by acting with $\gamma^{-1}$
  and is therefore by $\ast$-isomorphisms.
  Furthermore,
  as $g_\eta \ast g_\gamma$ for composable $\eta$ and $\gamma \in \G$
  is supported on a bisection and evaluates to one at $\eta\gamma$,
  it is a valid choice for $g_{\eta\gamma}$ so that the action is compatible
  with composition in $\G$.

  Finally, we show that the action is continuous.
  Let $\gamma_\lambda \to \gamma$ 
  and $\left[ a_\lambda \right]_{x_\lambda} \to \left[ a \right]_x$
  with $x_\lambda = s(\gamma_\lambda)$ 
  and $x = s(\gamma)$.
  We have to prove that 
  $\gamma_\lambda.\left[ a_\lambda \right]_{x_\lambda} \to \gamma.\left[ a \right]_x$.
  As the elements of $\intermediaryAlgebra$ 
  are exactly the continuous sections in the associated bundle,
  the convergence $\left[ a_\lambda \right]_{x_\lambda} \to \left[ a \right]_x$
  is equivalent to 
  $\left\| \left[ a_\lambda - a \right]_{x_\lambda} \right\|_{x_\lambda} \to 0$
  and it therefore suffices to show that 
  $\gamma_\lambda.\left[ a \right]_{x_\lambda} \to \gamma.\left[ a \right]_x$.
  We do this first for $a = \tilde\pi(f)$ 
  and then generalise to arbitrary $a$ as before.
  As $\tilde\pi$ is $\G$-equivariant we find that
  \[
    \gamma_\lambda.\left[ \tilde\pi(f) \right]_{x_\lambda}
    =
    \tilde\pi_{\gamma_\lambda.x_\lambda}\left( \gamma_\lambda.[ f ]_{x_\lambda} \right)
    \to
    \tilde\pi_{\gamma.x}\left( \gamma.\left[ f \right]_x \right)
    =
    \gamma.\left[ \tilde\pi(f) \right]_{x}.
  \]

  With the action by $\gamma$ being via \starhomomorphism{}, the same holds with $\tilde \pi(f)$
  replaced by a linear combination of finite products of this form.
  Since these are dense, convergence 
  $\gamma_\lambda.\left[ a_\lambda\right]_{x_\lambda} \to \gamma.\left[ a \right]_x$
  may be tested against such sections: 
  The net converges to $\gamma.\left[ a \right]_x$, if and only if
  for all $\eps>0,\,  n,k \in \N$, and $ f_{1,1},\ldots, f_{n,k} \in C( \tilde X )$
  with
  $\|\gamma.[a]_x - [ \sum\nolimits_{i=1}^k\tilde \pi(f_{1,i}) \cdots 
  \tilde \pi(f_{n,i}) ]_{\gamma.x} \|_{\gamma.x} < \eps$
  we eventually have 
  $\| \gamma_\lambda.[ 
  a_\lambda ]_{x_\lambda} 
  - [ \sum\nolimits_{i=1}^k\tilde \pi( f_{1,i} ) \cdots 
    \tilde \pi( f_{n,i} ) 
  ]_{\gamma_\lambda.x_\lambda} \|_{\gamma_\lambda.x_\lambda} 
  < \eps.$

  Let 
  $b = \sum\nolimits_{i=1}^k\tilde\pi\left( f_{1,i} \right) 
  \cdots \tilde\pi\left( f_{n,i} \right)$
  such that $[b]_{\gamma.x}$ approximates $\gamma.[a]_x$ up to $\eps$.
  Then 
  \begin{align*}
    \left\|
    \gamma_\lambda.\left[ a \right]_{x_\lambda}
    -
    \left[ b \right]_{\gamma_\lambda.x_\lambda}
    \right\|_{\gamma_\lambda.x_\lambda}
    &=
    \left\|
    \left[ a \right]_{x_\lambda}
    -
    \gamma_\lambda^{-1}.\left[ b \right]_{\gamma_\lambda.x_\lambda}
    \right\|_{x_\lambda}
    \\
    &\leq
    \left\|
    \left[ a \right]_{x_\lambda}
    -
    \left[ 
      \tilde\pi(g_\gamma)^\ast b \tilde\pi(g_\gamma) 
    \right]_{x_\lambda}
    \right\|_{x_\lambda}
    +
    \left\|
    \left[ 
      \tilde\pi(g_\gamma)^\ast b \tilde\pi(g_\gamma) 
    \right]_{x_\lambda}
    -
    \gamma_\lambda^{-1}.\left[ b \right]_{\gamma_\lambda.x_\lambda}
    \right\|_{x_\lambda}
    .
  \end{align*}
  By the arguments above the right-hand term vanishes as $\lambda\to\infty$
  while the left-hand term is the norm of a section and hence
  depends upper semi-continuously on $x_\lambda$.
  Hence
  \[
    \limsup\limits_{\lambda\to \infty}
    \left\|
    \gamma_\lambda.\left[ a \right]_{x_\lambda}
    -
    \left[ b \right]_{\gamma_\lambda.x_\lambda}
    \right\|_{\gamma_\lambda.x_\lambda}
    \leq
    \left\|
    \left[ a \right]_{x}
    -
    \left[ 
      \tilde\pi(g_\gamma)^\ast b \tilde\pi(g_\gamma) 
    \right]_{x}
    \right\|_{x}
    =
    \left\|
    [a]_x
    -
    \gamma^{-1}.\left[ b \right]_{\gamma.x}
    \right\|_x
    <\eps,
  \]
  so the action is continuous.
\end{proof}

We are now able to show that a groupoid $\G$ inherits the intersection property
from its boundary groupoid $\tilde \G$:
\begin{lemma}
  Let $\G$ be a locally compact Hausdorff groupoid with compact unit space and
  $\tilde \G$ its boundary groupoid.
  If $\tilde \G$ has the intersection property, then so does $\G$.
  \label{lem:intersectionPassesFromEnvelope}
\end{lemma}
\begin{proof}
  Let $\pi\colon \Csr(\G) \rightarrow \B(\H)$ be a representation 
  of $\Csr(\G)$ with $\ker(\pi) \cap C(X) = \{0\}$.
  By Arveson extension we may find a ucp extension 
  $\tilde \pi\colon \Csr(\tilde \G) \rightarrow \B(H)$
  and we denote by $D = C^\ast( \tilde \pi( \Csr(\tilde \G) ) )$ 
  the sub-$C^*$-algebra of $\B(\H)$ 
  generated by $\tilde \pi( \Csr( \G ) )$.
  We consider 
  $\intermediaryAlgebra \defeq C^{\ast}( \tilde \pi( C(\tilde X) ) ) 
  \subseteq D$ 
  as in Lemma~\ref{lem:extendUcpToGAlgebra}
  and endow it with the $\G$-algebra structure defined there.

  As $\intermediaryAlgebra$ contains $C(X)$ as a sub-$\G$-algebra, 
  we may find a ucp $\G$-map 
  $\phi\colon \intermediaryAlgebra \rightarrow C( \tilde X )$
  by extending the inclusion of $C( X )$.
  Then $\phi \circ \tilde\pi|_{C( \tilde X )}$ is a $\G$-map 
  fixing $C( X )$,
  hence by rigidity of $C( \tilde X )$ it is the identity.
  Therefore, both $\phi|_{\tilde \pi( C( \tilde X ) )}$ 
  and $\tilde\pi|_{C( \tilde X )}$ are isometries
  and since 
  the multiplicative domain of $\phi$ coincides with
  $\overline{\text{span}}{\{ u \mid \|u\|=1,\,\phi(u)\text{ unitary} \}}$
  and contains $\tilde \pi(u)$ for $u \in C( \tilde X )$ unitary
  which generate $\intermediaryAlgebra$ as a $\Cs$-algebra,
  $\phi$ is a \starhomomorphism{}.

  However, $\tilde \pi|_{C( \tilde X )}$ might fail to be
  a \starhomomorphism{}, 
  if $\phi$ has non-trivial kernel.
  This is alleviated as follows:
  We consider 
  $F= \overline{\ker(\phi) \cdot \tilde\pi( C_c( \tilde \G ) )} $
  and show that it is an ideal in $D$,
  or equivalently
  $\tilde \pi( C_c( \tilde \G ) ) \ker( \phi ) 
  \subseteq 
  \overline{ \ker( \phi ) \tilde \pi( C_c(\tilde \G) ) }$.
  Fix $a \in \ker(\phi)$, 
  as well as 
  $g \in C_c\left( \G \right)$ real-valued 
  and vanishing outside of an open bisection $B$.
  We want to show that $\pi(g)a$ is 
  contained in
  $\ker(\phi) \pi\left( C_c\left( \G \right) \right)$.
  Let $h = \sqrt[3]g$.
  Then $h^\ast \ast h$ is supported on the unit space $X$, 
  since $g$ is supported on a bisection
  and given by 
  \begin{equation*}
    \left( h^\ast \ast h \right)(x) = 
    \begin{cases}
      |h(\gamma_x)|^2 & x \in s(B)\\
      0
    \end{cases}
  \end{equation*}
  where $\gamma_x$ is the unique arrow in $B$ such that $Bx = \{\gamma_x\}$.
  Hence $h\ast h^\ast \ast h = g$
  and $\pi(g)a = \pi\left( h\ast h^\ast \ast h \right) a$.
  Since $h^\ast \ast h \in C(X)$ acts centrally on $\intermediaryAlgebra$,
  we find that 
  $\pi(g)a = \pi(h \ast h^\ast \ast h)a 
  = \left( \pi(h) \cdot a \cdot \pi(h)^\ast \right)\pi(h)$.

  We proceed to argue that $\phi(\pi(h) a \pi(h)^\ast)$ vanishes,
  or equivalently
  $\left[ \pi(h)a\pi(h)^\ast \right]_x \in \ker\left( \phi_x \right)$
  for all $x\in X$.
  First assume $x \not \in r(B)$,
  or $x \in r(B)$ with $xB = {\gamma}$ and $h(\gamma) = 0$.
  Then let $k = \sqrt[3] h$ and note that $k^\ast \ast k \in C_0\left(
  X\setminus x \right)$,
  so
  \begin{equation*}
    \left[ \pi(h) a \pi(h)^\ast \right]_x
    = 
    \left[ \pi(kk^\ast) \pi(k) a \pi(h)^\ast \right]_x
    =
    \left[ \pi(kk^\ast)\right]_x \left[ \pi(k) a \pi(h)^\ast \right]_x
    = 
    0.
  \end{equation*}
  On the other hand, if $x \in r(B)$ with $xB = {\gamma}$
  but $h(\gamma) \neq 0$,
  we may rescale $h$ to $h' \defeq h/h(\gamma)$.
  Then by definition of the $\G$-action 
  \begin{equation*}
    h(\gamma)^2 \cdot \gamma.\left[ a \right]_{s(\gamma)}
    =
    h(\gamma)^2 \left[ \pi(h') a \pi(h')^\ast \right]_x
    = 
    \left[ \pi(h) a \pi(h)^\ast \right]_x.
  \end{equation*}
  Now, since $\phi$ is a $\G$-map,
  $\gamma.[a]_{s(\gamma)}$ is in the kernel of $\phi_x$
  exactly if $[a]_{s(\gamma)}$ is in the kernel of $\phi_{s(\gamma)}$,
  which holds since $a\in \ker(\phi)$.
  So, for $g\in C_c(\G)$ real-valued and supported on a bisection,
  we find that
  \begin{equation*}
    \pi(g) a =
    \left( \pi(h)\cdot a \cdot \pi(h)^\ast \right) 
    \pi(h) \in \ker(\phi) \cdot \pi\left( C_c\left( \G \right) \right),
  \end{equation*}
  and as such $g$ span $C_c(\G)$ densely
  we may conclude that 
  $\pi( C_c( \G ) ) \ker(\phi) 
  \subseteq \overline{\ker(\phi) \pi( C_c( \G ) )}$.
  Note that $C_c( \G ) \cdot C( \tilde X )$ 
  spans $C_c( \tilde \G )$ densely
  and that $C( \tilde X )\ker( \phi ) \subseteq \ker( \phi )$,
  so that the collection of all $ga$ for $g$ and $a$ as above spans a dense subset of 
  $\tilde \pi( C_c( \tilde \G ) )\ker( \phi )$.
  Therefore $F$ is an ideal in $D$
  whose elements are exactly these fixed by left multiplication with an approximate unit 
  of $\ker(\phi) \subseteq \intermediaryAlgebra$.
  Hence $F\cap \intermediaryAlgebra = \ker(\phi)$.

  Denoting by $\Phi$ the quotient map $D \rightarrow D/F$,
  we consider the ucp map $\Phi \circ \tilde \pi$.
  As $\phi$ by design factors through the restriction of $\Phi$ to $\intermediaryAlgebra$,
  and $\phi \circ ( \Phi \circ \tilde \pi )|_{C\left( \tilde X \right)} 
  = \text{id}_{C\left( \tilde X \right)}$
  we may now conclude that 
  $( \Phi \circ \tilde \pi )|_{C\left( \tilde X \right)}$ 
  is a \starhomomorphism{ }
  since $\phi$ is an injective left inverse on $\intermediaryAlgebra / F \subseteq D/F$.
  Additionally, $( \Phi \circ \tilde \pi )|_{\Csr\left( \G \right)} = \Phi \circ \pi$
  is a \starhomomorphism{ },
  so both $C( \tilde X )$ and $\Csr( \G )$
  belong to the multiplicative domain of $\Phi \circ \tilde \pi$.
  As their product is dense in $\Csr( \tilde \G )$
  it follows that $\Phi \circ \tilde \pi$ itself is a \starhomomorphism{}.
  However, $\Phi \circ \tilde \pi$ is faithful on $C( \tilde X )$
  since 
  $\phi \circ ( \Phi \circ \tilde \pi )|_{C\left( \tilde X \right)} 
  = \text{id}_{C\left( \tilde X \right)}$,
  and by the intersection property of $\tilde \G$,
  it is itself faithful.
  As $\tilde \pi$ extends $\pi$,
  we may conclude that $\pi$ is faithful on $\Csr( \G )$.
  We conclude that $\G$ has the intersection property.
\end{proof}

Recall that Kawabe's characterisation \cite{Kawabe}
of \Cs{}-simplicity of discrete groups acting on compact spaces
generalises Kennedy's results \cite{KennedyIntrinsic}
by identifying it with the absence of \emph{recurrent} amenable subgroups
in the stabiliser subgroups of the action.
Given a discrete group $G$, 
both notions rely on endowing the set $\text{Sub}(G)$ of subgroups of $G$
equipped with the Chabauty topology, that is, 
the topology of pointwise convergence of indicator functions,
as well as the action of $G$ by conjugating subgroups.
A \emph{recurrent} subgroup of a group $G$ is then a subgroup $H$,
such that the closure of its orbit under the $G$-action
does not contain the trivial subgroup $\{e\}$.
The analogous notion in Kawabe's work, where $G$ acts on a compact space $X$,
considers amenable subgroups of the point stabilisers $\text{Stab}(x)$ of the action of $G$ on $X$,
equipped with an action of $G$ by conjugation.
To keep track of the basepoint, denote a subgroup $H \leq \text{Stab}(x)$
as $(x, H)$  in $X \times \text{Sub}(G)$
with the action of $g \in G$ by
$g.(x, H) = (g.x, gHg^{-1})$ and the product topology.
Such a subgroup is again \emph{recurrent},
if the closure of its orbit contains the trivial subgroup 
$(x, \{e\})$ at the same basepoint.
For $G$ acting minimally on $X$,
Kawabe shows that the absence of recurrent amenable subgroups 
in the stabilisers
is again equivalent to simplicity of the reduced crossed product $C(X) \rtimes G$,
while in general it is equivalent to the intersection property
of every closed, $G$-invariant subset of $X$.

In the following we provide a generalisation of one of Kawabe's results
\cite[Thm 5.2]{Kawabe} to \etale{} groupoids,
establishing a new sufficient criterion for the intersection property.
Let $\text{Sub}(\G)$ denote the space of all subgroups of the isotropy groups of $\G$, 
that is, the disjoint union of the $\text{Sub}(\G_x^x)$ 
ranging over all $x \in \Gzero$.
We endow $\text{Sub}(\G)$ with the subspace topology of the Chabauty topology 
of the power set of $\G$.
Recall that a subbasis of this topology is given by the collections
\begin{align*}
  \mathcal O_U &= \left\{ Y \subset \G \bigm| Y \cap U \neq \emptyset \right\}
  &&\text{and}&
  \mathcal O'_K &= \left\{ Y \subseteq \G \bigm| Y \cap K = \emptyset \right\}
\end{align*}
for $U\subseteq \G$ open
and $K\subseteq \G$ compact.
We may as before equip $\text{Sub}(\G)$ with an action of $\G$ by conjugation,
where an arrow $\gamma \in \G$ acts on a subgroup $\Lambda_{s(\gamma)}$
of $\G_{s(\gamma)}^{s(\gamma)}$
by conjugation $\gamma \Lambda_{s(\gamma)} \gamma^{-1}$.
\begin{definition} 
  Let $\G$ be a locally compact \etale{} Hausdorff groupoid
  and $x\in \Gzero$ a unit.  
  A subgroup $\Lambda \leq \G_x^x$ of the isotropy $\G_x^x$ at $x$ 
  is called \emph{recurrent}, 
  if the closure of the $\G$-orbit of $\Lambda$ under the conjugation action 
  does not contain the trivial subgroup $\{x\}$ of $\G_x^x$.  
\end{definition}
As in the setting of crossed products, the absence of \emph{amenable} recurrent subgroups
forces the boundary groupoid to be principal
and provides therefore a sufficient criterion for the intersection property of $\G$.
\begin{theorem}
  Let $\G$ be a Hausdorff \etale{} groupoid with compact unit space
  that does not have recurrent amenable subgroups in the isotropy,
  and in which the orbit of any unit in the groupoid contains at least two points.
  Then $\G$ has the intersection property.
  \label{thm:conditionSufficient}
\end{theorem}
\begin{proof}
  Let as before $\tilde \G$ be the boundary groupoid of $\G$, 
  let $Q\colon \tilde \G \rightarrow \G$ be the continuous groupoid homomorphism 
  given by $( \gamma, \tilde x) \mapsto \gamma$,
  and denote the unit spaces of $\G$ and $\tilde \G$ by $X$ and $\tilde X$, respectively.
  Assume that $\G$ has no amenable recurrent subgroups.
  Note that all isotropy groups $\tilde \G_{\tilde x}^{\tilde x}$ of the
  boundary groupoid are amenable by Proposition \ref{prop:amenableStabilisers}, 
  hence so are the corresponding subgroups $Q(\tilde \G_{\tilde x}^{\tilde x})$ of $\G$
  and by assumption none of them is recurrent.
  Let $\Lambda = \left\{ Q(\tilde \G_{\tilde x}^{\tilde x}) \mid \tilde x \in \tilde X \right\}$
  and note that its orbit closure $\overline{\G.\Lambda}$
  contains the trivial subgroups $\left\{ \left\{ x \right\} \mid  x \in X \right\}$
  at every unit of $\G$.
  However, $\Lambda$ is already invariant since
  $\gamma.\tilde \G_{\tilde x}^{\tilde x}
  = \tilde \G_{\gamma.\tilde x}^{\gamma.\tilde x}$
  and is furthermore closed:
  Let $\Phi\colon \tilde X \rightarrow \text{Sub}\left( \G \right)$
  be the isotropy map 
  $\tilde x \mapsto q(\tilde \G_{\tilde x}^{\tilde x})$.
  Then $\Lambda = \Phi(\tilde X)$ is the range of $\Phi$,
  and since $\tilde X$ is compact and $\text{Sub}(\G)$ Hausdorff,
  $\Lambda$ is closed provided that $\Phi$ is continuous.
  To verify this, we calculate
  \begin{align*}
    \Phi^{-1}\left( \mathcal O_U \right)
    &= \left\{ 
      \tilde x \bigm| \tilde \G_{\tilde x}^{\tilde x} \cap Q^{-1}(U) \neq \emptyset 
    \right\}
    &&\text{and}&
    \Phi^{-1}\left( \mathcal O'_K \right)
    &= \left\{ 
      \tilde x \bigm| \tilde \G_{\tilde x}^{\tilde x}\cap Q^{-1}(K) = \emptyset 
    \right\}\\
    &= s\left( \text{Iso}\left( \tilde \G \right) \cap Q^{-1}\left( U \right) \right)
    &&&
    &= \tilde X \setminus s\left( \text{Iso}\left( \tilde \G \right) \cap Q^{-1}(K) \right).
  \end{align*}
  As every orbit in $\G$ has at least two points, 
  $\text{Iso}(\tilde \G)$ is clopen by Proposition \ref{prop:clopenIsotropy},
  and furthermore $s$ is an open surjection
  while $Q^{-1}(U)$ and $Q^{-1}(K)$ are open and closed, respectively.
  Hence, $\Phi$ is continuous
  and $\Lambda = \Phi( \tilde X )$ is closed.
  Now, $\Lambda$ contains $\left\{ \left\{ x \right\} \mid  x \in X \right\}$,
  so for all $x \in X$ there is some $\tilde x \in \tilde X$
  with $q(\tilde x) = x$
  and $\tilde G_{\tilde x}^{\tilde x} = \left\{ \tilde x \right\}$.
  That is, there are enough trivial isotropies in $\tilde \G$
  to cover $X$ along $q$.
  Hence $Z \defeq \{ \tilde x \in \tilde X \mid \tilde \G_{\tilde x}^{\tilde x}
  = \{\tilde x\}\}$
  is a $\G$-invariant subset of $\tilde X$ with $q(Z) = X$.
  Furthermore $Z$ is closed, 
  since its complement
  $\tilde X \setminus Z = s( \text{Iso}( \tilde \G) \setminus \tilde X )$
  is the image of a closed set under the open, surjective map $s$.
  Thus $Z = \tilde X$ by Proposition \ref{prop:equivariantCover}, 
  so $\tilde \G$ is principal and hence has the intersection property.
  By Lemma \ref{lem:intersectionPassesFromEnvelope}
  we conclude that $\G$ inherits the intersection property from $\tilde \G$.
\end{proof}

\begin{remark}
  Note that as for crossed products, 
  the absence of amenable recurrent subgroups implies the intersection property
  not only for $\G$, but for every restriction of $\G$ to a closed invariant
  subset of $\Gzero$.
  In other words, absence of amenable recurrent subgroups implies the 
  \emph{residual intersection property} 
  (see \cite[Def~3.8]{BoenickeLi}).
\end{remark}

We finally apply Theorem \ref{thm:conditionSufficient}
to describe \Cs-simplicity of $\G$.
Recall, for example from \cite[Thm~A]{BoenickeLi},
that a groupoid is \Cs-simple
exactly if it has the intersection property and is minimal,
that is, if every orbit in $\Gzero$ is dense.
While minimality is a straightforward property of the groupoid, 
the intersection property is hard to describe 
without passing to the associated \Csalgebra{}s,
and is therefore in applications often replaced by the stronger 
assumption of topological principality.
See further the works by
Archbold--Spielberg \cite{ArchboldSpielbergTopFree},
Kawamura--Tomiyama \cite{KawamuraTopFree}, 
and Sierakowski \cite{SierakowskiIdeal} for crossed products
and B\"onicke--Li \cite{BoenickeLi} for groupoids.
Note that topological principality,
where a dense set of units in $\Gzero$ has trivial isotropy,
implies not having recurrent amenable subgroups in the isotropy,
since the choices of $\Lambda_x$ must be the trivial subgroup
for a dense set of units,
approximating any other trivial subgroup of the isotropy bundle.

\begin{corollary}
  Let $\G$ be a \emph{minimal} Hausdorff \etale{} groupoid with compact unit space
  that does not have recurrent amenable subgroups in the isotropy.
  Then $\G$ is \Cs{}-simple.
  \label{thm:sufficientForSimplicity}
\end{corollary}
\begin{proof}
  If $\Gzero$ consists of a single point
  then $\G$ is a discrete group
  and \Cs{}-simple by Kennedy's characterisation \cite[Thm 5.3]{KennedyIntrinsic}.
  If $\Gzero$ has more than one point,
  then every orbit has at least two points since it is dense.
  Therefore, $\G$ has the intersection property by Theorem \ref{thm:conditionSufficient}
  and is \Cs{}-simple since it is also minimal.
\end{proof}

\begin{question}
  Does the converse of Theorem \ref{thm:conditionSufficient} hold?
  That is, is the absence of recurrent amenable subgroups in the isotropy of $\G$
  \emph{equivalent} to the residual intersection property of $\G$?
\end{question}

\bibliographystyle{amsplain}
\bibliography{meta/mathsci,meta/nonmathsci}
\vspace{1cm}
\begin{minipage}[l]{\textwidth}
\noindent Clemens Borys\\
Department of Mathematical Sciences\\
University of Copenhagen\\ 
Universitetsparken 5, DK-2100, Copenhagen\\
Denmark \\
borys@math.ku.dk
\end{minipage}

\end{document}